\documentclass[12pt]{amsart}

\usepackage{etex}

\usepackage{epsf}
\usepackage{graphicx}
\usepackage{epstopdf}
\DeclareGraphicsExtensions{.eps}
\usepackage{latexsym}
\usepackage{amsfonts}
\usepackage{amscd}
\usepackage{amssymb}
\usepackage{amsmath}
\usepackage{amsthm}
\usepackage{bm}
\usepackage[margin=1.25in,dvips]{geometry}
\usepackage{xcolor}
\usepackage{pinlabel}
\usepackage{paralist}
\usepackage{enumitem}
\setlist{leftmargin=*}
\usepackage[all,cmtip]{xy}
\usepackage[color=blue!20!white,textsize=scriptsize]{todonotes}
\usepackage{caption,subcaption}

\usepackage{tikz}
\usetikzlibrary{arrows,positioning}
\tikzstyle{basepoint}=[circle,fill=black,minimum height=4pt,inner sep=0pt, outer sep=0pt, style={transform shape=false}]
\tikzstyle{crossing}=[circle,fill=white,minimum height=6pt,inner sep=0pt, outer sep=0pt, style={transform shape=false}]

\tikzstyle{blackdot}=[circle,fill=black,minimum height=6pt,inner sep=0pt, outer sep=0pt, style={transform shape=false}]
\tikzstyle{greendot}=[circle,fill=green!50!black,minimum height=6pt,inner sep=0pt, outer sep=0pt, style={transform shape=false}]
\tikzstyle{bluedot}=[circle,fill=blue,minimum height=6pt,inner sep=0pt, outer sep=0pt, style={transform shape=false}]
\tikzstyle{blackempty}=[circle,draw,thick,black,fill=white,minimum height=6pt,inner sep=0pt, outer sep=0pt, style={transform shape=false}]
\tikzstyle{greenempty}=[circle,draw,thick,green!50!black,fill=white,minimum height=6pt,inner sep=0pt, outer sep=0pt, style={transform shape=false}]

\definecolor{darkblue}{rgb}{0,0,0.4}
\usepackage{xr-hyper}
\usepackage[colorlinks=true, citecolor=darkblue, filecolor=darkblue, linkcolor=darkblue,urlcolor=darkblue]{hyperref}
\usepackage[all]{hypcap}
\PassOptionsToPackage{unicode}{hyperref}

\newtheorem{theorem}{Theorem}[section]
\newtheorem{lemma}[theorem]{Lemma}

\newtheorem{conjecture}[theorem]{Conjecture}

\newtheorem{proposition}[theorem]{Proposition}

\theoremstyle{definition}
\newtheorem{definition}[theorem]{Definition}

\newtheorem{remark}[theorem]{Remark}

\newtheorem{example}[theorem]{Example}

\def\R{\mathbb{R}}

\def\Z{\mathbb{Z}}
\def\Q{\mathbb{Q}}
\def\F{\mathbb{F}}

\def\HF {\widehat{\operatorname{HF}}}
\def\CFK {\widehat{\operatorname{CFK}}}
\def\HFK {\widehat{\operatorname{HFK}}}

\def\uCFK {\widetilde{\operatorname{CFK}}}
\def\uHFK {\widetilde{\operatorname{HFK}}}

\def\nextpt{\nu}

\def\CFKtil {\widetilde{\operatorname{CFK}}}
\def\HFKtil {\widetilde{\operatorname{HFK}}}

\def\ort {\mathfrak{o}}

\def\DD {\mathcal{D}}
\def\FF {\mathcal{F}}
\def\GG {\mathcal{G}}
\def\HH {\mathcal{H}}
\def\LL {\mathcal{L}}
\def\MM {\mathcal{M}}
\def\NN {\mathcal{N}}
\def\OO {\mathcal{O}}

\def\T{\mathbb{T}}

\def\a {\mathbf{a}}
\def\b {\mathbf{b}}
\def\c {\mathbf{c}}
\def\e {\mathbf{e}}
\def\p {\mathbf{p}}

\def\w {\mathbf{w}}
\def\x {\mathbf{x}}
\def\y {\mathbf{y}}
\def\z {\mathbf{z}}

\def\mults {\mathbf{n}}

\newcommand{\abs}[1] {\left\lvert #1 \right\rvert}

\def\minus{\smallsetminus}
\def\co{\colon\thinspace}

\DeclareMathOperator{\im}{im} \DeclareMathOperator{\id}{id} \DeclareMathOperator{\rank}{rank}

 \DeclareMathOperator{\nbd}{nbd}
 \DeclareMathOperator{\gr}{gr}

\DeclareMathOperator{\can}{can}

\definecolor{darkblue}{rgb}{0,0,0.5}
\definecolor{darkred}{rgb}{0.5,0,0}
\definecolor{darkgreen}{rgb}{0,0.5,0}

\numberwithin{equation}{section}

\renewcommand{\tilde}{\widetilde}
\renewcommand{\hat}{\widehat}

\newcommand{\from}{\colon}
\newcommand{\into}{\hookrightarrow}

\newcommand{\covers}{\gtrdot}
\newcommand{\iscovered}{\lessdot}

\newcommand{\diff}{d}
\newcommand{\bdy}{\partial}
\newcommand{\del}{\partial}
\DeclareMathOperator{\Id}{id}

\newcommand{\card}[1]{\left\vert{#1}\right\vert}

\newcommand{\mc}{\mathcal}

\newcommand{\Kh}{\operatorname{Kh}}
\newcommand{\rKh}{\widetilde{\Kh}}
\newcommand{\intgr}{\gr_{q}}
\newcommand{\homgr}{\gr_{h}}
\newcommand{\deltagr}{\gr_{\delta}}
\newcommand{\KhCx}{\operatorname{CKh}}

\newcommand{\rKhCx}{\widetilde{\KhCx}}

\newcommand{\Khdiff}{\diff_{\Kh}}
\newcommand{\KhAlg}[1][{}]{\mathbf{A}^{#1}}

\renewcommand{\th}{^{\text{th}}}
\renewcommand{\emptyset}{\varnothing}

\newcommand{\robar}{}

\begin{document}

\title{Khovanov homology and knot Floer homology for pointed links}

\title{Khovanov homology and knot Floer homology for pointed links}
\author{John A. Baldwin}
\address{Department of Mathematics, Boston College, Chestnut Hill, MA 02467}
\email{john.baldwin@bc.edu}
\thanks{JAB was supported by NSF grant DMS-1406383 and NSF CAREER grant DMS-1454865.}

\author{Adam Simon Levine}
\address{Department of Mathematics, Princeton University, Princeton, NJ 08544}
\email{asl2@math.princeton.edu}
\thanks{ASL was supported by NSF grant DMS-1405378.}

\author{Sucharit Sarkar}
\address{Department of Mathematics, Princeton University, Princeton, NJ 08544}
\email{sucharit@math.princeton.edu}

\thanks{SS was supported by NSF CAREER grant DMS-1350037.}

\date{\today}

\begin{abstract}
A well-known conjecture states that for any $l$-component link $L$ in $S^3$, the rank of the knot Floer homology of $L$ (over any field) is less than or equal to $2^{l-1}$ times the rank of the reduced Khovanov homology of $L$. In this paper, we describe a framework that might be used to prove this conjecture. We construct a modified version of Khovanov homology for links with multiple basepoints and show that it mimics the behavior of knot Floer homology. We also introduce a new spectral sequence converging to knot Floer homology whose $E_1$ page is conjecturally isomorphic to our new version of Khovanov homology; this would prove that the conjecture stated above holds over the field $\Z_2$.
\end{abstract}

\maketitle



\tableofcontents

\section{Introduction} \label{sec: introduction}

This paper concerns the following conjecture, first formulated in the case of knots by Rasmussen \cite{RasmussenHomologies}:
\begin{conjecture} \label{conj: Kh-HFK}
With coefficients in any field $\F$, for any $l$-component link $L \subset S^3$ equipped with a basepoint $p \in L$, we have
\begin{equation} \label{eq: Kh-HFK}
2^{l-1} \rank \rKh(L,p;\F) \ge \rank \HFK(L;\F).
\end{equation}
\end{conjecture}
\noindent Here $\rKh(L,p)$ denotes the reduced Khovanov homology of $L$, and $\HFK(L)$ denotes the knot Floer homology of $L$. (We shall frequently suppress $p$ from the notation.) Conjecture \ref{conj: Kh-HFK} would yield a new proof that Khovanov homology detects the unknot (first shown by Kronheimer and Mrowka \cite{KronheimerMrowkaUnknot}); it would also imply that Khovanov homology detects the trefoil knot, which is the only knot for which $\HFK$ has rank $3$ \cite[Corollary 8]{HeddenWatsonGeography}. In this paper, we shall lay out a new framework for studying Conjecture \ref{conj: H(X,D0)-intro} via a modified version of Khovanov homology for links with multiple basepoints.

The starting point for most known inequalities between Khovanov homology and other link invariants is the skein exact sequence. Suppose that $A$ is a functorial link invariant which agrees with reduced Khovanov homology for unlinks and for elementary merge and split cobordisms between them, and which satisfies a skein exact sequence: if $L_0$ and $L_1$ are the resolutions of $L$ at a crossing, as shown in Figure \ref{fig: resolutions}, then the cobordism maps relating $A(L)$, $A(L_0)$, and $A(L_1)$ fit into an exact sequence of the form
\begin{equation} \label{eq: A-skein}
\dots \to A(L) \to A(L_0) \to A(L_1) \to A(L) \to \dots,
\end{equation}
just like the skein sequence satisfied by Khovanov homology. Then for any link $L$, one should expect to obtain a spectral sequence whose $E_2$ page is isomorphic to $\rKh(L)$ and which converges to $A(L)$, yielding an inequality
\[
\rank \rKh(L) \ge \rank A(L).
\]
Starting with the work of Ozsv\'ath and Szab\'o \cite{OSzDouble} (where $A(L)$ is $\HF(-\Sigma(L);\Z_2)$, the Heegaard Floer homology of the branched double cover of $L$ with coefficients in $\Z_2$), this technique has been used to prove rank inequalities between Khovanov homology and a host of other knot invariants, including the monopole and instanton Floer homologies of the branched double cover \cite{BloomMonopole, ScadutoOdd}, instanton knot homology \cite{KronheimerMrowkaUnknot}, and many others (e.g. \cite{DaemiAbelian, SzaboSpectral}).

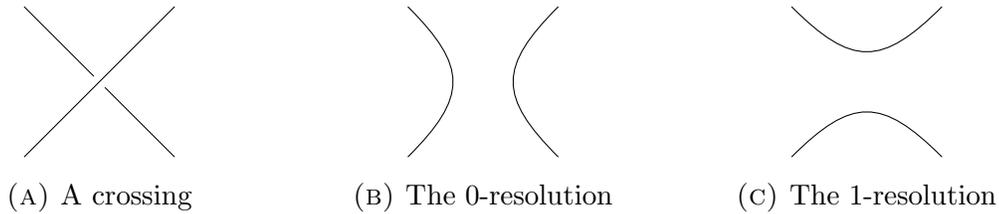
\begin{figure}
  \centering
  \begin{subfigure}[b]{0.3\textwidth}
    \[
    \begin{tikzpicture}[scale=0.2]
      \draw (0,10) -- (10,0);
      \node[crossing] at (5,5) {};
      \draw (0,0) -- (10,10);
    \end{tikzpicture}
    \]
    \caption{A crossing}
  \end{subfigure}
  \hspace{0.1in}
  \begin{subfigure}[b]{0.3\textwidth}
    \[
    \begin{tikzpicture}[scale=0.2]
      \draw (0,0) .. controls (4,4) and (4,6) .. (0,10);
      \draw (10,0) .. controls (6,4) and (6,6) .. (10,10);
    \end{tikzpicture}
    \]
    \caption{The $0$-resolution}
  \end{subfigure}
  \hspace{0.1in}
  \begin{subfigure}[b]{0.3\textwidth}
    \[
    \begin{tikzpicture}[scale=0.2]
      \draw (0,0) .. controls (4,4) and (6,4) .. (10,0);
      \draw (0,10) .. controls (4,6) and (6,6) .. (10,10);
    \end{tikzpicture}
    \]
    \caption{The $1$-resolution}
  \end{subfigure}
\caption{The $0$- and $1$-resolutions of a crossing.}
\label{fig: resolutions}
\end{figure}

The main obstacle to making knot Floer homology fit into this framework is that it behaves quite differently for links than do other invariants. For instance, if $L$ is the Hopf link, then $\HFK(L)$ has rank $4$, whereas $\rKh(L)$ has rank $2$ (as do the other invariants mentioned above). This necessitates the factor of $2^{l-1}$ in \eqref{eq: Kh-HFK}. Moreover, $\HFK$ cannot satisfy a skein sequence as in \eqref{eq: A-skein}: both resolutions at a crossing in a minimal diagram for the Hopf link are unknots, so $\rank \HFK(L_0) = \rank \HFK(L_1)=1$, which would violate the triangle inequality. The explanation for this difficulty is that knot Floer homology is really an invariant of \emph{non-degenerate pointed links}.

\begin{definition}\label{def:pointed-link}
A \emph{pointed link} $\LL = (L,\p)$ is a link $L$ together with a finite set of points $\p$ on $L$. We say that $\LL$ is \emph{non-degenerate} if every component of $L$ contains at least one point of $\p$. 
\end{definition}

Given a non-degenerate pointed link $\LL=(L,\p)$, the knot Floer homology of $\LL$, denoted $\HFKtil(\LL)$ in the notation of \cite{BaldwinLevineSpanning}, can be understood as the sutured Floer homology of the complement of $L$ equipped with a pair of meridional sutures for each point of $\p$. The original invariant $\HFK(L)$ (as defined in \cite{OSzKnot} and reinterpreted in \cite{OSzLink}) is simply $\HFKtil(\LL)$ when $\p$ is chosen to contain exactly one point on each component of $L$. On the other hand, if $p_0 \in \p$ is on the same component of $L$ as some other point of $\p$, then setting $\LL' = (L, \p \minus \{p_0\})$,
\begin{equation} \label{eq: HFK-doubling}
\HFKtil(\LL) \cong \HFKtil(\LL') \otimes V,
\end{equation}
where $V \cong \Z \oplus \Z$ with generators in bigradings $(0,0)$ and $(-1,-1)$. In particular, for an $l$-component pointed link $\LL$,
\begin{equation} \label{eq: HFK-tildehat}
\HFKtil(\LL) \cong \HFK(L) \otimes V^{\otimes \abs{\p}-l}.
\end{equation}

Manolescu \cite{ManolescuSkein} proved that $\HFKtil$ (with coefficients in $\Z_2$) does satisfy a appropriate skein sequence: if $\p$ is a set of points disjoint from the crossing being resolved, such that $\LL = (L,\p)$, $\LL_0 = (L_0,\p)$, and $\LL_1 = (L_1,\p)$ are all non-degenerate, then there is an exact sequence
\begin{equation} \label{eq: HFK-skein}
\cdots \to \HFKtil(-\LL;\Z_2) \to \HFKtil(-\LL_0;\Z_2) \to \HFKtil(-\LL_1;\Z_2) \to \HFKtil(-\LL;\Z_2) \to \cdots,
\end{equation}
where $-\LL$ denotes the mirror of $\LL$.\footnote{See Section \ref{subsec: cube-construction} for a discussion of orientations. Additionally, note that Wong \cite{WongSkein} has extended the skein sequence to $\Z$ coefficients using grid diagrams.} The number of factors of $V$ coming from \eqref{eq: HFK-tildehat} varies in each term in \eqref{eq: HFK-skein}. This eliminates the difficulties discussed above. For the example of the Hopf link, with one marked point on each component, the three groups now have ranks $4$, $2$, and $2$, since the resolutions $L_0$ and $L_1$ are now each unknots with two marked points. More generally, the skein sequence implies that for a quasi-alternating link $L$ with $l$ components,
\begin{equation} \label{eq: HFK-QA}
\rank \HFK(L;\Z_2) = 2^{l-1} \det(L) = 2^{l-1} \rank \rKh(L;\Z_2).
\end{equation}
As a result, \eqref{eq: Kh-HFK} is an equality when $L$ is quasi-alternating. On the other hand, when $L$ is an $l$-component unlink, $\rank \HFK(L;\F) = \rank \rKh(L;\F) = 2^{l-1}$ over any field $\F$, so \eqref{eq: Kh-HFK} is far from sharp.

For a pointed link $\LL = (L,\p)$, let $\hat\LL$ denote the split union of $\LL$ with an unknot containing a single marked point. Note that $\HFKtil(\hat\LL)$ is isomorphic to the direct sum of two copies of $\HFKtil(\LL)$, with Maslov gradings shifted by $\pm\frac12$ (see Lemma \ref{lemma: HFK-unred} below). In particular, over any field $\F$, we have
\begin{equation} \label{eq: HFK-unreduced}
\rank \HFKtil(\hat\LL;\F) = 2^{\abs{\p} -l+1} \rank \HFK(L;\F).
\end{equation}
By analogy to Khovanov homology, we may think of $\HFKtil(\hat\LL)$ as the ``unreduced knot Floer homology'' of $\LL$.

In view of Conjecture \ref{conj: Kh-HFK}, in this paper we introduce a variant of Khovanov homology for pointed links that mimics the properties of knot Floer homology. Given a diagram for a pointed link $(L,\p)$ (which is not required to be non-degenerate), we associate a chain complex $\KhCx(L,\p)$, which as a group is isomorphic to $\Lambda_{\p} \otimes \KhCx(L)$, where $\Lambda_\p$ is the exterior algebra on generators $\{y_p \mid p \in \p\}$. The differential combines the original Khovanov differential with a contribution coming from the basepoint action studied in Section \ref{subsec: Kh-action}; see Section \ref{subsec: mapping-cone-def} for the definition. Denote the homology of $\KhCx(L,\p)$ by $\Kh(L,\p)$.

The following theorem summarizes the main results of Sections \ref{sec: khovanov} and \ref{sec: HFK}:

\begin{theorem} \label{thm: Kh-properties}
Let $\LL = (L,\p)$ be a pointed link in $S^3$, where $\p \ne \emptyset$.
\begin{enumerate}
\item \label{item: Kh-properties: invariance}
The chain homotopy type of $\KhCx(L,\p)$ (as a differential $\Lambda_{\p}$--module) is an invariant of the isotopy type of $(L,\p)$ and does not depend on the choice of diagram; thus, $\Kh(L,\p)$ is a pointed link invariant.

\item \label{item: Kh-properties: doubling}
If $\p$ contains a point $p$ that is on the same component of $L$ as some other point of $\p$, then
\begin{equation} \label{eq: Kh-doubling}
\Kh(L,\p) \cong \Kh(L,\p \minus \{p\}) \otimes V.
\end{equation}

\item \label{item: Kh-properties: reduced}
With coefficients in any field $\F$, and for each point $p \in \p$, we have
\begin{equation} \label{eq: Kh-reduced}
\rank \Kh(L,\p; \F) \le 2^{\abs{\p}} \rank \rKh(L,p;\F).
\end{equation}
When $L$ is a knot this relation is an equality.

\item \label{item: Kh-properties: unlink}
If $L$ is an $l$-component unlink, then $\Kh(L,\p)$ is naturally isomorphic to $\HFKtil(\hat \LL)$, with rank $2^l$. Moreover, under these isomorphisms, the maps on $\Kh$ associated to elementary merges and splits of planar unlinks agree (up to sign) with the corresponding maps on $\HFKtil$. (See Proposition \ref{prop: HFK-edgemaps0} for a precise statement.)
\end{enumerate}
\end{theorem}

\begin{remark}
Note that Conjecture \ref{conj: Kh-HFK} would follow if we could show that the rank inequality
\begin{equation} \label{eq: Kh-HFK-pointed}
\rank \Kh(L,\p;\F) \ge \rank \HFKtil(\hat\LL;\F).
\end{equation}
holds for any non-degenerate pointed link $\LL = (L,\p)$.
In particular, the power of $2$ in \eqref{eq: Kh-HFK} would be a consequence of \eqref{eq: HFK-unreduced} and \eqref{eq: Kh-reduced}.
\end{remark}

We now review the construction that is used to prove the various spectral sequences discussed above before seeing how it must be modified in the present setting.
\begin{definition}\label{defn:cubical-chain-complex}
Let $(C,\diff)$ be a chain complex equipped with a direct sum
decomposition along the vertices of a cube $\{0,1\}^n$,
\[
C=\bigoplus_{v\in\{0,1\}^n}C_v,
\]
so that the component of the differential from $C_u$ to $C_v$, call it
$\diff^{u,v}$, is non-zero only if $u\leq v$ with respect to the
product partial order on the cube $\{0,1\}^n$.  Then we say $C$ is a
\emph{chain complex with a cubical filtration} (or \emph{cubical complex} for short), where the filtration $\FF$ is the height in the cube; namely, $C_v$ lies in filtration
level $\card{v}$.

Associated to this filtration is a spectral sequence converging to
$H_*(C,\diff)$ in at most $n$ pages. The $(E_0,d_0)$ page is
\[
\bigoplus_{v\in\{0,1\}^n}(C_v,\diff^{v,v}).
\]
Therefore, the $E_1$ page is
\[
\bigoplus_{v\in\{0,1\}^n}H_*(C_v,\diff^{v,v})
\]
and the $d_1$ differential is a sum of terms corresponding to the edges of the cube, i.e. pairs $(u,v)$ with $u \le v$ and $\abs{v} = \abs{u}+1$; we denote this condition by $u \lessdot v$. The component of the $d_1$ differential from $H_*(C_u,\diff^{u,u})$ to $H_*(C_v,\diff^{v,v})$ is the map induced on homology by $\diff^{u,v}$.
\end{definition}

Let $L \subset S^3$ be a link presented by a diagram with $n$ crossings labeled $c_1, \dots, c_n$. For each $v \in \{0,1\}^n$, let $L_v$ denote the planar unlink obtained by taking the $v_i$ resolution of $c_i$, according to the rule in Figure \ref{fig: resolutions}, and let $l_v$ be the number of components in $L_v$. The Khovanov complex $\KhCx(L)$ and the reduced version $\rKhCx(L)$ (defined below in Section \ref{sec: khovanov}) are both cubical complexes in which the only nonzero differentials correspond to pairs $u \lessdot v$.

For a link invariant $A$ as discussed above, the key step for constructing a spectral sequence from $\rKh(L)$ to $A(L)$ is to iterate the construction of the skein sequence \eqref{eq: A-skein}. This produces a cubical complex $C = \bigoplus_{v \in \{0,1\}^n} C_v$ with homology isomorphic to $A(L)$, such that $H_*(C_v) = A(L_v) = \rKh(L_v)$ and the maps on homology induced by $\diff^{u,v}$ for $u \lessdot v$ agree with the edge maps in the Khovanov cube. In the spectral sequence associated to $C$, the $E_1$ page is then isomorphic (as a chain complex) to $\rKhCx(L)$, and therefore the $E_2$ page is isomorphic (as a group) to $\rKh(L)$.

In \cite{BaldwinLevineSpanning}, we (the first and second authors) showed how to iterate Manolescu's skein sequence in the same manner as with other invariants. Let $\LL = (L,\p)$ be an $l$-component pointed link represented by a diagram in which every edge contains a point of $\p$. Adapting the results of \cite{BaldwinLevineSpanning}, we obtain a cubical chain complex $(X,D)$ with
\begin{equation} \label{eq: H(X,D)-intro}
H_*(X,D) \cong \HFKtil(-\hat\LL;\Z_2),
\end{equation}
where $D$ counts holomorphic polygons in a Heegaard multi-diagram that encodes all the resolutions $\LL_v$. The $E_1$ page of the associated spectral sequence is
\[
\bigoplus_{v \in \{0,1\}^n} \HFKtil(-\hat\LL_v;\Z_2),
\]
with a completely explicit $d_1$ differential. While we originally hoped to prove that the $E_2$ page would be isomorphic to an appropriate multiple of $\rKh(L)$, leading to a proof of Conjecture \ref{conj: Kh-HFK}, this turns out to be false; indeed, the $E_2$ page is not even a link invariant. (See \cite[Remark 7.7]{BaldwinLevineSpanning}.) Thus, the analogy with previously known spectral sequences breaks down.

Instead, in this paper we introduce a new filtration on $X$ that makes use of the internal Alexander gradings on the summands of $X$. The differential $D^0$ in the associated graded complex of this filtration counts only some of the holomorphic polygons that contribute to $D$. The induced spectral sequence implies that
\begin{equation} \label{eq: G-spectral}
\rank H_*(X,D^0) \ge \rank H_*(X,D).
\end{equation}
Moreover, $(X,D^0)$ is also a cubical complex whose $E_1$ page agrees with that of $(X,D)$ as a group; however, the $d_1$ differentials differ.

We now state our main (and perhaps surprising) conjecture:
\begin{conjecture} \label{conj: H(X,D0)-intro}
The complex $(X,D^0)$ is quasi-isomorphic to $\KhCx(L,\p)$ via a map that 
respects the cubical filtrations. As a result,
\begin{equation} \label{eq: H(X,D0)-intro}
H_*(X, D^0) \cong \Kh(L,\p;\Z_2).
\end{equation}
\end{conjecture}
Observe that Conjecture \ref{conj: H(X,D0)-intro}, together with equations \eqref{eq: H(X,D)-intro}, \eqref{eq: G-spectral}, would immediately imply that Conjecture \ref{conj: Kh-HFK} holds over $\Z_2$, via equation \eqref{eq: Kh-HFK-pointed}. (See Conjecture \ref{conj: H(X,D0)} for a more precise version that incorporates gradings.)

The primary evidence for Conjecture \ref{conj: H(X,D0)-intro} is the following:
\begin{theorem} \label{thm: E1iso-intro}
The $E_1$ pages of the spectral sequences associated to the cubical filtrations on $(X, D^0)$ and $\KhCx(L,\p;\Z_2)$ are isomorphic chain complexes. Therefore, the $E_2$ pages are isomorphic as $\Z_2$--vector spaces.
\end{theorem}
Roughly, this theorem follows from using item \eqref{item: Kh-properties: unlink} of Theorem \ref{thm: Kh-properties} to identify the edge maps in the two chain complexes. A careful proof is given in Section \ref{sec: cube}.

To prove Conjecture \ref{conj: H(X,D0)-intro}, it would suffice to produce a filtered chain map $\Phi$ from $(X,D^0)$ to $\KhCx(\LL;\Z_2)$ that induces the isomorphism of $E_1$ pages given by Theorem \ref{thm: E1iso-intro}. Standard facts about spectral sequences then imply that $\Phi$ induces isomorphisms on all subsequent pages. However, to write down such a map, one would presumably need to fully understand $(X,D^0)$ on the chain level, which involve difficult counts of holomorphic polygons.

At present, we do not know how to prove that the cube of resolutions for knot Floer homology holds with coefficients in $\Z$, due to technical issues regarding the orientations of moduli spaces of holomorphic polygons. Nevertheless, we work with $\Z$ coefficients for much of the paper --- most notably Propositions \ref{prop: Kh-unlink}, \ref{prop: Kh-edgemaps}, \ref{prop: HFK-unlink}, and \ref{prop: HFK-edgemaps0}, which together provide the identification described in item \eqref{item: Kh-properties: unlink} of Theorem \ref{thm: Kh-properties}. If the cube of resolutions can be proved over $\Z$, we anticipate that these results will easily imply the analogue of Theorem \ref{thm: E1iso-intro} in that setting.

\begin{remark} \label{rmk: odd}
It is instructive to consider a related link invariant, \emph{odd Khovanov homology}, defined by Ozsv\'ath, Rasmussen, and Szab\'o \cite{OSzR-kh-oddkhovanov}. Given a link $L$, one associates reduced and unreduced chain complexes $\KhCx'(L)$ and $\rKhCx{}'(L)$, whose $\Z_2$ reductions are the same as those of $\KhCx(L)$ and $\rKhCx(L)$, respectively. The homologies of these complexes are link invariants, denoted $\Kh'(L)$ and $\rKh{}'(L)$. These invariants can be very different from their ``even'' analogues $\Kh(L)$ and $\rKh(L)$. Most notably, the analogue of Conjecture \ref{conj: Kh-HFK} for odd Khovanov homology is false. For instance, if $K$ is the $(3,4)$ torus knot ($8_{19}$), then $\rank \rKh(K;\Q) = \rank \HFK(K;\Q) = 5$, while $\rank \rKh{}'(K;\Q) = 3$ \cite[Section 5]{OSzR-kh-oddkhovanov}.\footnote{There is no known knot for which $\rank \rKh{}'(K;\Q) > \rank \rKh(K;\Q)$; Shumakovitch has verified this for all knots through 16 crossings. However, many examples are known over finite fields; for instance, if $K = 9_{46}$, then $\rank \rKh(K;\Z_3) = 9$, while $\rank \rKh{}'(K;\Z_3)=11$ \cite{ShumakovitchPatterns}.} Although there are ways to define an odd analogue of $\Kh(L,\p)$ (see Remark \ref{rmk: roberts} below), one should not expect this to be related to an object in knot Floer homology.
\end{remark}

\subsection*{Acknowledgments} We are grateful to Matthew Hedden, Peter Ozsv\'ath, Alexander Shumakovitch, and Zolt\'an Szab\'o for helpful discussions.

\section{Khovanov homology for pointed links} \label{sec: khovanov}

In this section, we will do a brief review of the Khovanov chain
complex, describe the modified version for pointed links, and prove some of its properties. Unless otherwise mentioned, we will follow the following
conventions.
\begin{enumerate}[leftmargin=*]
\item All chain complexes are freely and finitely generated over $\Z$,
  and are \emph{bigraded}; the first grading is called the homological
  grading $\homgr$, and the second grading is called the quantum
  grading $\intgr$. The differential raises the bigrading by $(1,0)$.

\item For any two such bigraded chain complexes $A$ and $B$, the
  \emph{tensor product} $A\otimes B$ is a bigraded chain complex with following
  differential:
  \[
  \diff(a\otimes b)=\diff(a)\otimes b+(-1)^{\homgr(a)}a\otimes\diff(b).
  \]
\item For any such bigraded chain complex $C$, and for any
  $(r,s)\in\Z^2$, the \emph{shifted complex} $\Sigma^{(r,s)}C$ is the
  complex
  \(
  S^{(r,s)}\otimes C
  \)
  where $S^{(r,s)}$ is a single copy of $\Z$ supported in bigrading
  $(r,s)$. The differential on $\Sigma^{(r,s)} C$ is therefore equal to $(-1)^r$ times the differential on $C$.

\item
A homogeneous map $f \co C \to C'$ of bidegree $(r,s)$ should really be understood as a degree-0 map $f \co C \to \Sigma^{(-r,-s)} C'$. Thus, we call $f$ a \emph{chain map} if $f \circ \diff_C = \diff_{\Sigma^{(-r,-s)} C'} \circ f$, i.e., if $f \circ d_C = (-1)^r \diff_{C'} \circ f$. Likewise, a \emph{chain homotopy} $H$ between two such chain maps $f_1,f_2$ must satisfy
\begin{align}
\label{eq: homotopy-sign}
f_1 - f_2 &= H \circ \diff_C + \diff_{\Sigma^{(-r,-s)} C'} \circ H \\
\nonumber &= H \circ \diff_C + (-1)^r \diff_{C'} \circ H.
\end{align}

\end{enumerate}

\subsection{The Khovanov complex}
Let $L$ be an oriented, $l$-component link diagram in $\R^2$ with $n$ crossings $c_1,\dots,c_n$. The \emph{Kauffman cube of resolutions}
\cite{Kau-knot-resolutions} is constructed as follows: To a vertex
$v=(v_1,\dots,v_n)\in\{0,1\}^n$ of the cube, associate the complete
resolution $L_v$ of the link diagram $L$, by resolving the $i\th$
crossing $c_i$
$\vcenter{\hbox{\begin{tikzpicture}[scale=0.04]
\draw (0,10) -- (10,0);
\node[crossing] at (5,5) {};
\draw (0,0) -- (10,10);
\end{tikzpicture}}}$
by the \emph{$0$-resolution}
$\vcenter{\hbox{\begin{tikzpicture}[scale=0.04]
\draw (0,0) .. controls (4,4) and (4,6) .. (0,10);
\draw (10,0) .. controls (6,4) and (6,6) .. (10,10);
\end{tikzpicture}}}$
if $v_i=0$ or by the \emph{$1$-resolution}
$\vcenter{\hbox{\begin{tikzpicture}[scale=0.04]
\draw (0,0) .. controls (4,4) and (6,4) .. (10,0);
\draw (0,10) .. controls (4,6) and (6,6) .. (10,10);
\end{tikzpicture}}}$
otherwise, for all $1\leq i\leq n$. Let $l_v$ denote the number of components of $L_v$.

The \emph{Khovanov complex} $\KhCx(L)$ is obtained by applying a $(1+1)$-dimensional TQFT to this cube of resolutions \cite{Kho-kh-categorification}. Specifically, consider the two dimensional Frobenius algebra $\KhAlg=\Z[x]/x^2$ with comultiplication
$\KhAlg\to\KhAlg\otimes\KhAlg$ given by $1\mapsto 1\otimes x+x\otimes
1$ and $x\mapsto x\otimes x$. The chain group $\KhCx(L)$ is defined as the direct
sum
\[
\KhCx(L)=\bigoplus_{v\in\{0,1\}^n}\KhCx(L_v).
\]
where for each planar unlink $L_v$, the group $\KhCx(L_v)$ is a tensor product
of copies of $\KhAlg$, one copy for each component of the planar
unlink $L_v$. That is,
\[
\KhCx(L_v) = \bigotimes_{\pi_0(L_v)}\KhAlg \cong \Z[x_1,\dots,x_{l_v}]/(x_1^2,\dots,x_{l_v}^2),
\]
where $\pi_0(L_v)$ denotes the components of the planar unlink $L_v$;
the second isomorphism comes from numbering the components
$C_1,\dots,C_{l_v}$ and letting $\Z[x_i]/x_i^2$ denote the copy of
$\KhAlg$ corresponding to $C_i$.

We define several gradings on $\KhCx(L)$. The \emph{homological grading} $\homgr$ of the summand $\KhCx(L_v)$ is $\card{v}-n_-$, where $\card{v}$ denotes the \emph{height} of $v$ in
the cube, namely, its $L^1$-norm $\sum_i v_i$, and $n_-$ is the number
of negative crossings
$\vcenter{\hbox{\begin{tikzpicture}[scale=0.04]
\draw[->] (0,10) -- (10,0);
\node[crossing] at (5,5) {};
\draw[->] (0,0) -- (10,10);
\end{tikzpicture}}}$
in the link diagram $L$. The \emph{quantum grading} $\intgr$ is defined as follows: On $\KhAlg$, define $\intgr(1)=0$ and $\intgr(x)=-2$. On each summand $\KhCx(L_v) = \bigotimes_{\pi_0(L_v)}\KhAlg$, extend the grading
multiplicatively on the tensor product, and add the constant
$(l_v + \abs{v} + n - 3n_-)$. It is easy to verify that the
quantum grading has the same parity as the number of link components of $L$. Finally, the \emph{delta grading} $\deltagr$ is defined by $\deltagr = \homgr - \frac12 \intgr$; it takes values in $\Z + \frac{l}{2}$.

The \emph{differential} $\Khdiff\from\KhCx\to\KhCx$ is more
involved. There is a partial order on $\{0,1\}^n$ by declaring $u \leq
v$ if $u_i\leq v_i$ for all $1\leq i\leq n$. Write $u \iscovered v$ if
$u<v$ and $\card{v}-\card{u}=1$; such pairs correspond to the edges of
the cube. The Khovanov differential decomposes along the edges,
namely, the component $\Khdiff^{u,v}$ of $\Khdiff$ that goes from the
summand $\KhCx(L_u)$ to the summand $\KhCx(L_v)$ is non-zero only if
$u\iscovered v$.  For $u\iscovered v$, let $\bar{\imath} \in \{1,\dots,n\}$ be
the unique value such that $u_{\bar{\imath}}< v_{\bar{\imath}}$;
define the \emph{sign assignment}
\[s_{u,v}=\sum_{i=1}^{\bar{\imath}-1}v_i\pmod 2.\] The planar unlink
$L_v$ is gotten from the planar unlink $L_u$ either by \emph{merging}
two circles into one, or by \emph{splitting} one circle into two. The
component
\[
\Khdiff^{u,v}\from\bigotimes_{\pi_0(L_u)}\KhAlg\to \bigotimes_{\pi_0(L_v)}\KhAlg
\]
is defined to be $(-1)^{s_{u,v}+n_-}$ times the Frobenius
multiplication (respectively, comultiplication) map on the relevant
factors if $L_v$ is obtained from $L_u$ by a merge (respectively, a
split), extended by the identity map $\Id$ on the remaining
factors. It is straightforward to check that $\Khdiff$ is homogeneous of bigrading $(1,0)$ and satisfies $(\Khdiff)^2=0$. Khovanov proved that the (bigraded) chain homotopy type of $(\KhCx(L), \Khdiff)$ is an invariant of the underlying link, and not just the link diagram. Its homology is the \emph{Khovanov homology} $\Kh(L)$. The delta grading gives a decomposition
\[
\Kh(L) = \bigoplus_{\delta \in \Z + \frac{l}{2}} \Kh^\delta(L).
\]

\subsection{The basepoint actions} \label{subsec: Kh-action}
Now let us fix a \emph{checkerboard coloring} of the link-diagram $L$; for
concreteness, unless we explicitly declare otherwise, we will choose
the checkerboard coloring where the unbounded region is colored
white. We define the \emph{parity} of each edge $e$ of $L$: $e$ is
\emph{even} if the orientation of $e$ as the boundary of the adjacent
black region agrees with the orientation that $e$ inherits from that
of $L$, and \emph{odd} otherwise. Note that two edges on opposite
sides of a crossing must have opposite parity. For any point $p$ on
$L$ away from the crossings, let $\epsilon(p) \in \Z/2$ denote the
parity of the edge containing $p$.

For any point $p$ on $L$, away from the crossing, we define the
\emph{basepoint action} $\xi_p$ on $\KhCx(L)$ as follows. Place a
small planar unknot $U$ near $p$, disjoint from $L$. We may merge this
circle with $L$ to get a link diagram that is isotopic to $L$, with
the isotopy supported in a small neighborhood of $p$:
$\vcenter{\hbox{\begin{tikzpicture}[scale=0.04]
\draw (0,0) .. controls (4,4) and (4,6) .. (0,10);
\node[basepoint] (p) at (3.2,5) {};
\node[left=0pt of p] {\tiny $p$};
\draw (12,5) circle (5);
\node at (12,5) {\tiny $U$};
\end{tikzpicture}}}
\to
\vcenter{\hbox{\begin{tikzpicture}[scale=0.04]
\draw (0,0) .. controls (4,5) and (7,0) .. (12,0);
\draw (0,10) .. controls (4,5) and (7,10) .. (12,10);
\draw (12,0) arc (-90:90:5);
\end{tikzpicture}}}$.
Using the Frobenius multiplication, we thus have a merge map
\[
\mu\co \KhCx(L \amalg U) = \KhCx(L) \otimes \KhAlg \to \KhCx(L),
\]
and we define $\xi_p = (-1)^{\epsilon(p)} \mu(\cdot, x)$. Explicitly,
$\xi_p$ acts on each summand $\KhCx(L_v)$ by multiplication by
$(-1)^{\epsilon(p)} x$ in the tensor factor corresponding to the
component of $L_v$ that contains $p$, extended by the identity on the
other factors. Note that $\xi_p$ is a chain map with
$(\homgr,\intgr)$-grading $(0,-2)$, and it satisfies
$\xi_p^2=0$. Moreover, for any points $p,p'$, we have $\xi_p \circ
\xi_{p'} = \xi_{p'} \circ \xi_p$.

It follows that for $m$ basepoints $\p = \{p_1, \dots, p_m\}$ on
$L$, away from the crossings, $\KhCx(L)$ becomes a dg-module over the
ring
\[
D_{\p} = \Z[\xi_{p_1}, \dots, \xi_{p_m}]/(\xi_{p_1}^2, \dots, \xi_{p_m}^2).
\]
\begin{lemma}\label{lem:indep-of-checkerboard-coloring}
The two different checkerboard colorings for $L$ produce isomorphic
complexes over $D_{\p}$.
\end{lemma}
\begin{proof}
  The quantum grading $\intgr$ comes to our aid. Define a chain map
  $\KhCx(L)\to\KhCx(L)$ as $x\mapsto (-1)^{(\intgr(x)+|L|)/2}x$. This
  is an automorphism of $\KhCx(L)$, and it intertwines the two
  different module structures of $\KhCx(L)$ over $D_{\p}$ coming from
  the two different checkerboard colorings.
\end{proof}

\begin{lemma}\label{lemma:dp-action-away-from-basepoints}
  Assume two pointed link diagrams $(L,\p)$ and $(L,\p')$ are related
  by a Reidemeister move away from the basepoints $\p$. That is,
  assume there is a disk $\mathbb{D}^2\subset\R^2$ not containing any
  of the basepoints of $\p$, so that $L$ and $L'$ agree outside
  $\mathbb{D}^2$, and differ inside $\mathbb{D}^2$ as one of the
  following:
  \[
  \vcenter{\hbox{\begin{tikzpicture}[scale=0.07]
    \draw[dotted] (0,0) circle (10);
    \clip (0,0) circle (10);
    \draw (-10,-10) -- (-5,0) .. controls (0,10) and (5,5) .. (5,0);
    \node[crossing] at (-5,0) {};
    \draw (-10,10) -- (-5,0) .. controls (0,-10) and (5,-5) .. (5,0);
  \end{tikzpicture}}}
  \leftrightarrow
  \vcenter{\hbox{\begin{tikzpicture}[scale=0.07]
    \draw[dotted] (0,0) circle (10);
    \clip (0,0) circle (10);
    \draw (-10,-10) .. controls (-5,0) .. (-10,10);
  \end{tikzpicture}}}
  \qquad\qquad
  \vcenter{\hbox{\begin{tikzpicture}[scale=0.07]
    \draw[dotted] (0,0) circle (10);
    \clip (0,0) circle (10);
    \draw (-10,-10) -- (-5,0) .. controls (-2,6) and (2,6) .. (5,0) -- (10,-10);
    \node[crossing] at (-5,0) {};
    \node[crossing] at (5,0) {};
    \draw (-10,10) -- (-5,0) .. controls (-2,-6) and (2,-6) .. (5,0) -- (10,10);
  \end{tikzpicture}}}
  \leftrightarrow
  \vcenter{\hbox{\begin{tikzpicture}[scale=0.07]
    \draw[dotted] (0,0) circle (10);
    \clip (0,0) circle (10);
    \draw (-10,-10) .. controls (-5,0) and (5,0) .. (10,-10);
    \draw (-10,10) .. controls (-5,0) and (5,0) .. (10,10);
  \end{tikzpicture}}}
  \qquad\qquad
  \vcenter{\hbox{\begin{tikzpicture}[scale=0.07]
    \draw[dotted] (0,0) circle (10);
    \clip (0,0) circle (10);
    \draw (-10,-10) -- (10,10);
    \node[crossing] at (0,0) {};
    \draw (-10,10) -- (10,-10);
    \node[crossing] at (-4.5,-4.5) {};
    \node[crossing] at (4.5,-4.5) {};
    \draw (-10,0) .. controls (-5,0) and (-7,-2) .. (-4.5,-4.5)
    .. controls (-2,-7) and (2,-7)  ..
    (4.5,-4.5) .. controls (7,-2) and (5,0) .. (10,0);
  \end{tikzpicture}}}
  \leftrightarrow
  \vcenter{\hbox{\begin{tikzpicture}[scale=0.07]
    \draw[dotted] (0,0) circle (10);
    \clip (0,0) circle (10);
    \draw (-10,-10) -- (10,10);
    \node[crossing] at (0,0) {};
    \draw (-10,10) -- (10,-10);
    \node[crossing] at (-4.5,4.5) {};
    \node[crossing] at (4.5,4.5) {};
    \draw (-10,0) .. controls (-5,0) and (-7,2) .. (-4.5,4.5)
    .. controls (-2,7) and (2,7)  ..
    (4.5,4.5) .. controls (7,2) and (5,0) .. (10,0);
  \end{tikzpicture}}}
  \]
  Then Khovanov's chain homotopy equivalence between $\KhCx(L)$ and
  $\KhCx(L')$  (as constructed in~\cite{Kho-kh-categorification}) is a
  chain homotopy equivalence over $D_\p$.
\end{lemma}

\begin{proof}
The fact that the maps $\xi_p$ commute with the chain homotopy equivalences associated to Reidemeister moves is stated in \cite[Section 3]{Kho-kh-patterns}. Roughly, the homotopy equivalence between $\KhCx(L)$ and $\KhCx(L')$ is obtained by a sequence of ``cancellations'' of acyclic subcomplexes corresponding to the following local pictures:
  \begin{enumerate}[leftmargin=*]
  \item Merging a circle labeled $1$:
    $\vcenter{\hbox{\begin{tikzpicture}[scale=0.04]
      \draw (0,0) .. controls (4,4) and (4,6) .. (0,10);
      \draw (12,5) circle (5);
      \node at (12,5) {\tiny $1$};
    \end{tikzpicture}}}\to
      \vcenter{\hbox{\begin{tikzpicture}[scale=0.04]
        \draw (0,0) .. controls (4,5) and (7,0) .. (12,0);
        \draw (0,10) .. controls (4,5) and (7,10) .. (12,10);
        \draw (12,0) arc (-90:90:5);
    \end{tikzpicture}}}$.
  \item Splitting off a circle labeled $x$:
    $      \vcenter{\hbox{\begin{tikzpicture}[scale=0.04]
        \draw (0,0) .. controls (4,5) and (7,0) .. (12,0);
        \draw (0,10) .. controls (4,5) and (7,10) .. (12,10);
        \draw (12,0) arc (-90:90:5);
    \end{tikzpicture}}}\to\vcenter{\hbox{\begin{tikzpicture}[scale=0.04]
      \draw (0,0) .. controls (4,4) and (4,6) .. (0,10);
      \draw (12,5) circle (5);
      \node at (12,5) {\tiny $x$};
    \end{tikzpicture}}}$.
  \end{enumerate}
  These circles that are being merged or split off lie entirely within
  the disk $\mathbb{D}^2$; in fact, they are the following four
  circles
 \[
  \vcenter{\hbox{\begin{tikzpicture}[scale=0.07]
    \draw[dotted] (0,0) circle (10);
    \clip (0,0) circle (10);
    \draw[ultra thin] (-10,-10) -- (-5,0) .. controls (0,10) and (5,5) .. (5,0);
    \node[crossing] at (-5,0) {};
    \draw[ultra thin] (-10,10) -- (-5,0) .. controls (0,-10) and (5,-5) .. (5,0);
    \draw[thick] (-4,0) .. controls (-4,1.5) .. (-2.5,4) .. controls (0,8) and (5,5) .. (5,0);
    \draw[thick] (-4,0) .. controls (-4,-1.5) .. (-2.5,-4) .. controls (0,-8) and (5,-5) .. (5,0);
  \end{tikzpicture}}}
  \qquad\qquad
  \vcenter{\hbox{\begin{tikzpicture}[scale=0.07]
    \draw[dotted] (0,0) circle (10);
    \clip (0,0) circle (10);
    \draw[ultra thin] (-10,-10) -- (-5,0) .. controls (-2,6) and (2,6) .. (5,0) -- (10,-10);
    \node[crossing] at (-5,0) {};
    \node[crossing] at (5,0) {};
    \draw[ultra thin] (-10,10) -- (-5,0) .. controls (-2,-6) and (2,-6) .. (5,0) -- (10,10);
    \draw[thick] (-4,0) .. controls (-4,6) and (4,6) .. (4,0);
    \draw[thick] (-4,0) .. controls (-4,-6) and (4,-6) .. (4,0);
  \end{tikzpicture}}}
  \qquad\qquad
  \vcenter{\hbox{\begin{tikzpicture}[scale=0.07]
    \draw[dotted] (0,0) circle (10);
    \clip (0,0) circle (10);
    \draw[ultra thin] (-10,-10) -- (10,10);
    \node[crossing] at (0,0) {};
    \draw[ultra thin] (-10,10) -- (10,-10);
    \node[crossing] at (-4.5,-4.5) {};
    \node[crossing] at (4.5,-4.5) {};
    \draw[ultra thin] (-10,0) .. controls (-5,0) and (-7,-2) .. (-4.5,-4.5)
    .. controls (-2,-7) and (2,-7)  ..
    (4.5,-4.5) .. controls (7,-2) and (5,0) .. (10,0);
    \draw[thick] (-3.5,-5) .. controls (-2,-7) and (2,-7) .. (3.5,-5);
    \draw[thick] (-3,-3) .. controls (0,0) and (0,0) .. (3,-3);
    \draw[thick] (3,-3) .. controls (4,-4) .. (3.5,-5);
    \draw[thick] (-3,-3) .. controls (-4,-4) .. (-3.5,-5);
  \end{tikzpicture}}}
  \qquad\qquad
  \vcenter{\hbox{\begin{tikzpicture}[scale=0.07]
    \draw[dotted] (0,0) circle (10);
    \clip (0,0) circle (10);
    \draw[ultra thin] (-10,-10) -- (10,10);
    \node[crossing] at (0,0) {};
    \draw[ultra thin] (-10,10) -- (10,-10);
    \node[crossing] at (-4.5,4.5) {};
    \node[crossing] at (4.5,4.5) {};
    \draw[ultra thin] (-10,0) .. controls (-5,0) and (-7,2) .. (-4.5,4.5)
    .. controls (-2,7) and (2,7)  ..
    (4.5,4.5) .. controls (7,2) and (5,0) .. (10,0);
    \draw[thick] (-3.5,5) .. controls (-2,7) and (2,7) .. (3.5,5);
    \draw[thick] (-3,3) .. controls (0,0) and (0,0) .. (3,3);
    \draw[thick] (3,3) .. controls (4,4) .. (3.5,5);
    \draw[thick] (-3,3) .. controls (-4,4) .. (-3.5,5);
  \end{tikzpicture}}}
  \]
  appearing in the left pictures for the Reidemeister I and II moves,
  and in both pictures for the Reidemeister III move.  (See, e.g.,  \cite{Bar-kh-khovanovs}.) Consequently,
  they do not contain any of the basepoints from $\p$, and therefore,
  the chain homotopy equivalences work over $D_\p$.
\end{proof}

\begin{lemma} \label{lemma:movepoint} If $p_1$ and $p_2$ are two
  basepoints in $\p$ that lie on the same component of $L$, then
  $\xi_{p_1}$ and $\xi_{p_2}$ are chain homotopic via a
  $(-1,-2)$-graded chain homotopy $H_{\alpha}$, depending only on the
  choice of an arc $\alpha$ connecting $p_1$ and $p_2$ on the
  underlying link; and the homotopy $H_{\alpha}$ commutes with
  $\xi_{p}$ for all $p$ in $\p$.
\end{lemma}

\begin{proof}
  This is essentially \cite[Proposition 2.2]{HN-kh-unlink-detection},
  but with signs. It is enough to construct the homotopy when $p_1$
  and $p_2$ are on the two sides of a single crossing; in the general
  case, given an arc $\alpha$ from $p_1$ to $p_2$, we may move $p_1$
  to $p_2$ along $\alpha$, one crossing at a time, and add the
  homotopies from the single crossing case.

  So assume $p_1$ and $p_2$ are on the opposite sides of a single
  crossing. There are two cases: they could be on the overpass
$\vcenter{\hbox{\begin{tikzpicture}[xscale=0.04, yscale=0.04]
\draw (0,10) -- (10,0);
\node[crossing] at (5,5) {};
\draw (0,0) -- (10,10);
\node[basepoint] at (0.5,0.5) {};
\node[basepoint] at (9.5,9.5) {};
\end{tikzpicture}}}$
  or the underpass
$\vcenter{\hbox{\begin{tikzpicture}[scale=0.04]
\draw (0,10) -- (10,0);
\node[crossing] at (5,5) {};
\draw (0,0) -- (10,10);
\node[basepoint] at (0.5,9.5) {};
\node[basepoint] at (9.5,0.5) {};
\end{tikzpicture}}}$.
The same homotopy works for both cases, so we will tackle the two
cases simultaneously. Assume we have four basepoints
$p_1,q_1,p_2,q_2$ arranged around a crossing $c$ as follows:
\[
\begin{tikzpicture}[scale=0.08]
\draw (0,10) -- (10,0);
\node[crossing] at (5,5) {};
\draw (0,0) -- (10,10);
\node[basepoint] (p1) at (0.5,0.5) {};
\node[basepoint] (p2) at (9.5,9.5) {};
\node[basepoint] (q1) at (9.5,0.5) {};
\node[basepoint] (q2) at (0.5,9.5) {};
\node[left=0pt of p1] {$p_1$};
\node[right=0pt of q1] {$q_1$.};
\node[right=0pt of p2] {$p_2$};
\node[left=0pt of q2] {$q_2$};
\end{tikzpicture}
\]
Each of the maps $(\xi_{p_1}-\xi_{p_2})$ and $(\xi_{q_1}-\xi_{q_2})$
agree, up to signs, with the following map
$f\from\KhCx(L)\to\KhCx(L)$.
\begin{enumerate}
\item If in the complete resolution, the two arcs near $c$ are in the
  same component,
  \[
  \vcenter{\hbox{
      \begin{tikzpicture}[scale=0.04]
        \draw (0,0) .. controls (4,4) and (4,6) .. (0,10);
        \draw (10,0) .. controls (6,4) and (6,6) .. (10,10);
        \draw[dotted,thick] (0,10) .. controls (-2,16) and (12,16) .. (10,10);
        \draw[dotted,thick] (0,0) .. controls (-2,-6) and (12,-6) .. (10,0);
      \end{tikzpicture}}}
  \text{ or }
  \vcenter{\hbox{
      \begin{tikzpicture}[scale=0.04]
        \draw (0,0) .. controls (4,4) and (6,4) .. (10,0);
        \draw (0,10) .. controls (4,6) and (6,6) .. (10,10);
        \draw[dotted,thick] (0,10) .. controls (-6,12) and (-6,-2) .. (0,0);
        \draw[dotted,thick] (10,10) .. controls (16,12) and (16,-2) .. (10,0);
      \end{tikzpicture}}}
  \]
  then $f$ is the map $\Z[x]/x^2\to\Z[x]/x^2$
  \[
  1\mapsto 2x\qquad x\mapsto 0,
  \]
  extended by identity on the other factors.
\item If in the complete resolution, the two arcs near $c$ are in
  different components,
  \[
  \vcenter{\hbox{
      \begin{tikzpicture}[scale=0.04]
        \draw (0,0) .. controls (4,4) and (4,6) .. (0,10);
        \draw (10,0) .. controls (6,4) and (6,6) .. (10,10);
        \draw[dotted,thick] (0,10) .. controls (-6,12) and (-6,-2) .. (0,0);
        \draw[dotted,thick] (10,10) .. controls (16,12) and (16,-2) .. (10,0);
      \end{tikzpicture}}}
  \text{ or }
  \vcenter{\hbox{
      \begin{tikzpicture}[scale=0.04]
        \draw (0,0) .. controls (4,4) and (6,4) .. (10,0);
        \draw (0,10) .. controls (4,6) and (6,6) .. (10,10);
        \draw[dotted,thick] (0,10) .. controls (-2,16) and (12,16) .. (10,10);
        \draw[dotted,thick] (0,0) .. controls (-2,-6) and (12,-6) .. (10,0);
      \end{tikzpicture}}}
  \]
  then $f$ is the map $\Z[x_1,x_2]/(x_1^2,x_2^2)\to\Z[x_1,x_2]/(x_1^2,x_2^2)$
  \[
  1\mapsto x_1+x_2\qquad x_1\mapsto x_1x_2\qquad x_2\mapsto
  x_1x_2\qquad x_1x_2\mapsto 0,
  \]
  extended by identity on the other factors.
\end{enumerate}

Define the homotopy $H_c\from\KhCx(L)\to\KhCx(L)$ summand-wise. The
component $H_c^{u,v}$ from the summand $\KhCx(L_u)$ to the summand
$\KhCx(L_v)$ is non-zero only if $L_u$ is obtained from $L_v$ by
changing the resolution at $c$ from $0$ to $1$ (therefore, in
particular, $u\covers v$); and in that case, $H_c^{u,v}$ is defined to
be $(-1)^{s_{v,u}}$ times the Frobenius multiplication (respectively,
comultiplication) map on the relevant factors if $L_v$ is obtained
from $L_u$ by a merge (respectively, a split), extended by the
identity map $\Id$ on the remaining factors. It is fairly
straightforward to check that
\[
\Khdiff H_c+H_c\Khdiff=f.\qedhere
\]
\end{proof}

\begin{remark}
  In Lemma~\ref{lemma:movepoint}, if $p_1=p_2$ and the arc $\alpha$
  connecting $p_1$ and $p_2$ on $L$ is chosen to be
  the entire link component that contains $p_1$, then the
  corresponding homotopy satisfies
  \[
  \Khdiff H_\alpha+H_\alpha\Khdiff=\xi_{p_1}-\xi_{p_2}=0.
  \]
  Therefore, the map $(-1)^{\homgr(\cdot)}H_{\alpha}(\cdot)$ is a
  chain map $\KhCx(L)\to\KhCx(L)$. This is closely related to the
  Batson--Seed higher differential on
  $\KhCx(L)$~\cite{batson-seed-2015}.
\end{remark}

\begin{remark}
Note that Lemma \ref{lemma:movepoint} does not imply that the chain homotopy type of $\KhCx(L)$ as a $D_\p$--module is invariant under isotopies of $\p$ along $L$. We do not know whether this stronger form of invariance holds except in the case where $\p$ consists of a single point $p$, where it was proved by Khovanov \cite{Kho-kh-patterns} as follows. Instead of moving a strand past $p$, one may perform a sequence of Reidemeister moves avoiding $p$ together with moving strands around the point at infinity. The latter move changes the checkerboard coloring; however, by
Lemma~\ref{lem:indep-of-checkerboard-coloring}, this does not affect the isomorphism type of $\KhCx(L)$ over $D_{\p}$. However, this proof does not generalize to the case where $\abs{\p}>1$.
\end{remark}



Given a point $p_0 \in L$, the \emph{reduced Khovanov complex}
$\rKhCx(L,p_0)$ is defined to be the kernel of $\xi_{p_0}$, with the
$(\homgr,\intgr)$-grading shifted by $(0,1)$:
\[
\rKhCx(L,p_0)=\Sigma^{(0,1)}\ker(\xi_{p_0}).
\]
This complex is (canonically) isomorphic to the cokernel of $\xi_{p_0}$, with the bigrading shifted by $(0,-1)$. The chain homotopy type of $\rKhCx(L,p_0)$ is an invariant of the underlying pointed link, and its homology is denoted
$\rKh(L,p_0)$, or simply $\rKh(L)$ if $p_0$ is understood from
context. Just like the unreduced homology, the reduced homology decomposes according to $\delta$ gradings:
\[
\rKh(L) = \bigoplus_{\delta \in \Z + \frac{l+1}{2}} \rKh{}^\delta(L).
\]
Note that for any collection $\p$ of basepoints on $L$, the
basepoint actions restrict to give an action of $D_\p$ on
$\rKhCx(L,p_0)$. If $p\in\p$ is on the same component of $L$ as $p_0$,
then by Lemma~\ref{lemma:movepoint}, the restriction of $\xi_p$ to
$\rKhCx(L,p_0)$ is nulhomotopic.

\subsection{The Khovanov complex for a pointed link}\label{subsec: mapping-cone-def}


Given a pointed link $(L,\p)$ with $n$ crossings and $m$ basepoints,
let $\Lambda_\p=\Lambda^*(y_p \mid p \in \p)$ be the exterior algebra
on generators $\{y_p \mid p \in \p\}$. Define the
$(\homgr,\intgr)$-bigrading of each $y_p$ to be $(1,2)$; define the
bigrading on $\Lambda_\p$ by extending this multiplicatively. (It follows that $\Lambda_\p$ is supported in delta grading $0$.) We define
\[
\KhCx(L,\p) = \Lambda_\p\otimes\KhCx(L),
\]
equipped with the differential
\begin{equation} \label{eq: Kh-diff} \Khdiff(a\otimes
  b)=(-1)^{\homgr(a)}a \otimes \Khdiff(b) +\sum_{p \in \p} (y_p\wedge a)\otimes\xi_p(b),
\end{equation}
which has bigrading $(1,0)$, and hence shifts the delta grading by $1$.
We call $\KhCx(L,\p)$ the \emph{Khovanov complex} of the pointed link
$(L,\p)$. It is the Koszul complex (iterated mapping cone) of the maps
$\{\xi_p \mid p \in \p\}$. When $\p=\emptyset$, the complex is simply
the Khovanov complex of $L$.

Note that $\KhCx(L,\p)$ carries a left action by $\Lambda_\p$, given
by $y_p(a\otimes b)=(y_p\wedge a)\otimes b$, since
\begin{align*}
&\phantom{==}\Khdiff(y_p(a\otimes b))+y_p(\Khdiff(a\otimes b))\\
&=\big((-1)^{1+\homgr(a)}(y_p\wedge a)\otimes\Khdiff(b)+\sum_{q \in
  \p} (y_q\wedge y_p\wedge a)\otimes\xi_q(b)\big)\\
&\qquad\qquad+\big((-1)^{\homgr(a)}(y_p\wedge a) \otimes \Khdiff(b) +\sum_{q \in \p} (y_p\wedge
  y_q\wedge a)\otimes\xi_q(b)\big)\\
&=0.
\end{align*}


In more explicit terms, the complex $\KhCx(L,\p)$ can be viewed in
relation to the two cubes $\{0,1\}^m$ and $\{0,1\}^n$ as
follows. Assume the points of $\p$ are labeled $\{p_1, \dots,
p_m\}$. For $u\in\{0,1\}^m$, let $1\leq i_1 <\dots < i_{\card{u}} \leq
m$ be the indices so that $u_{i_j}=1$ for all $1\leq j\leq \card{u}$;
and let $y_u$ denote the monomial
$y_{p_{i_{1}}}\wedge\dots\wedge y_{p_{i_{\card{u}}}} \in \Lambda_\p$. As
a group, $\KhCx(L,\p)$ is a direct sum of groups along the vertices
$(u,v)\in\{0,1\}^m\times\{0,1\}^n$: the chain group over $(u,v)$ is
$y_u\otimes\KhCx(L_v)$. The differential is along the edges of the
cube $\{0,1\}^m\times\{0,1\}^n$: if $v'\covers v$, the component for
the differential from $(u,v)$ to $(u,v')$ is
$(-1)^{\card{u}}\Khdiff^{v,v'}$, namely, $(-1)^{\card{u}}$ times the
component of the Khovanov differential from $\KhCx(L_{v})$ to
$\KhCx(L_{v'})$; and if $u'\covers u$, the component for the
differential from $(u,v)$ to $(u',v)$ is the map
\[
y_u\otimes b \mapsto
(-1)^{s_{u,u'}}y_{u'}\otimes\xi_{p_{\bar\imath}}(b),
\]
where $\bar\imath\in\{1,\dots,m\}$ is the unique value so that $u'_{\bar\imath}>u_{\bar\imath}$.

Note, $\KhCx(L,\p)$ restricted to the sub-cube $\{u\}\times\{0,1\}^n$
is simply the Khovanov complex $\KhCx(L)$, up to a shift of the
bigradings. On the other hand, $\KhCx(L,\p)$ restricted to the
sub-cube $\{0,1\}^m\times\{v\}$ is the Khovanov complex
$\KhCx(L_v,\p)$ of the planar unlink $L_v$ marked with basepoints
$\p$, up to a shift of bigradings and changes in the signs of the
differentials: There is a global change
$(-1)^{\card{v}+n_-}$ in the sign coming from the shift, and there are
local sign changes in the $\xi_p$--actions since (irrespective of how
one orients $L_v$) the parity of the edge in $L$ containing $p$ could
be different from the parity of the edge in $L_v$ containing $p$. We
will discuss this more in Section~\ref{subsec: kh-cube-filtration}.

\begin{example}
Let $(K,\p)$ denote the single-crossing unknot
\[
\begin{tikzpicture}[scale=0.08]
\draw (0,0) .. controls (3,0) and (7,10) .. (10,10);
\node[crossing] at (5,5) {};
\draw[-open triangle 60] (0,10) .. controls (3,10) and (7,0) .. (10,0);
\draw (0,10) arc (90:270:5);
\node[basepoint] (p1) at (-5,5) {};
\node[basepoint] (p2) at (15,5) {};
\node[left=0pt of p1] {$p_1$};
\node[right=0pt of p2]  {$p_2$};
\draw (10,0) arc (-90:90:5);
\end{tikzpicture}
\]
with two marked points $p_1, p_2$, oriented as shown (so that $p_1$ is
on an odd edge, and $p_2$ is on an even edge). Let $C$ denote the
single circle in the $0$-resolution of $K$; and let $C_1,C_2$ denote
the two circles in the $1$-resolution of $K$ with $p_i\in C_i$ for
$i\in\{1,2\}$. Write $\KhCx(C)=\Z[x]/x^2$, and
$\KhCx(C_i)=\Z[x_i]/x_i^2$ for $i\in\{1,2\}.$ The Khovanov complex
$\KhCx(K)$ is
\[
\xymatrix@C=20ex{
\Z[x]/x^2\ar[r]^-{1\mapsto -(x_1+x_2),\,\, x\mapsto -x_1x_2}&\Z[x_1,x_2]/(x_1^2,x_2^2);
}
\]
and the Khovanov complex $\KhCx(K,\p)$ is

\def\transxa{-5}\def\transya{-7}
\def\transxb{5}\def\transyb{-3.5}
\def\transxc{15}
\[
\scriptsize
\begin{tikzpicture}[scale=0.5]

\node (y00x0) at (0,0) {$y_{00}\Z[x]/x^2$};
\node (y10x0) at (\transxa,\transya) {$y_{10}\Z[x]/x^2$};
\node (y01x0) at (\transxb,\transyb) {$y_{01}\Z[x]/x^2$};
\node (y11x0) at (\transxa+\transxb,\transya+\transyb) {$y_{11}\Z[x]/x^2$};
\node (y00x1) at (\transxc,0) {$y_{00}\Z[x_1,x_2]/(x_1^2,x_2^2)$};
\node (y10x1) at (\transxa+\transxc,\transya) {$y_{10}\Z[x_1,x_2]/(x_1^2,x_2^2)$};
\node (y01x1) at (\transxb+\transxc,\transyb) {$y_{01}\Z[x_1,x_2]/(x_1^2,x_2^2)$};
\node (y11x1) at (\transxa+\transxb+\transxc,\transya+\transyb) {$y_{11}\Z[x_1,x_2]/(x_1^2,x_2^2)$};

\draw[->] (y00x0) -- (y00x1) node[midway,anchor=south] {\tiny $y_{00}\mapsto
  -y_{00}(x_1+x_2)$}
node[midway,anchor=north] {\tiny $y_{00}x\mapsto -y_{00}x_1x_2$};
\draw[->] (y00x0) -- (y10x0) node[midway,anchor=east] {\tiny
  $y_{00}\mapsto -y_{10}x $};
\draw[->] (y00x0) -- (y01x0) node[midway,anchor=west] {\tiny
  $y_{00}\mapsto y_{01}x $};
\draw[->] (y01x0) -- (y01x1) node[pos=0.3,anchor=south] {\tiny $y_{01}\mapsto
  y_{01}(x_1+x_2)$}
node[pos=0.3,anchor=north] {\tiny $y_{01}x\mapsto y_{01}x_1x_2$};
\draw[->] (y01x0) -- (y11x0) node[pos=0.2,anchor=east] {\tiny
  $y_{01}\mapsto -y_{11}x $};
\draw[->] (y00x1) -- (y01x1) node[pos=0.6,anchor=south
west,align=left,inner sep=0pt] {\tiny
  $y_{00}\mapsto y_{01}x_2$\\$y_{00}x_1\mapsto y_{01}x_1x_2$};
\draw[->] (y01x1) -- (y11x1) node[midway,anchor=north
west,align=left,inner sep=0pt] {\tiny
  $y_{01}\mapsto -y_{11}x_1$\\$y_{01}x_2\mapsto -y_{11}x_1x_2$};
\draw[->] (y10x1) -- (y11x1) node[pos=0.1,anchor=north east,align=right,inner sep=0pt] {\tiny
  $y_{10}\mapsto -y_{11}x_2$\\$y_{10}x_1\mapsto -y_{11}x_1x_2$};
\draw[->] (y10x0) -- (y11x0) node[midway,anchor=east] {\tiny
  $y_{10}\mapsto -y_{11}x $};
\draw[->] (y11x0) -- (y11x1) node[pos=0.3,anchor=south] {\tiny $y_{11}\mapsto
  -y_{11}(x_1+x_2)$}
node[pos=0.3,anchor=north] {\tiny $y_{11}x\mapsto -y_{11}x_1x_2$};

\draw[line width=8pt, white] (y10x0) -- (y10x1);
\draw[->] (y10x0) -- (y10x1) node[midway,anchor=south,fill=white,outer
sep=1pt] {\tiny $y_{10}\mapsto
  y_{10}(x_1+x_2)$}
node[midway,anchor=north,fill=white,outer
sep=1pt] {\tiny $y_{10}x\mapsto y_{10}x_1x_2$};
\draw[line width=8pt, white] (y00x1) -- (y10x1);
\draw[->] (y00x1) -- (y10x1) node[pos=0.6,anchor=north
west,align=left, inner sep=0pt] {\tiny
  $y_{00}\mapsto -y_{10}x_1$\\$y_{00}x_2\mapsto -y_{10}x_1x_2$};

\end{tikzpicture}
\]
 (Here, we have suppressed the tensor product notation, and written
 $ab$ instead of $a\otimes b$.)  The homology is four-dimensional,
 generated by
 \begin{align*}
   y_{00}x_2&\in\Lambda_\p\Z[x_1,x_2]/(x_1^2,x_2^2)&\text{in grading $(0,-1)$,}\\
   y_{10}x_2&\in\Lambda_\p\Z[x_1,x_2]/(x_1^2,x_2^2)&\text{in grading $(1,1)$,}\\
   -y_{11}+(y_{10}+y_{01})&\in\Lambda_\p\Z[x]/x^2\oplus
   \Lambda_\p\Z[x_1,x_2]/(x_1^2,x_2^2)&\text{in grading $(1,3)$,}\\
   y_{11}&\in\Lambda_\p\Z[x_1,x_2]/(x_1^2,x_2^2)&\text{in grading $(2,5)$.}
 \end{align*}
\end{example}

\begin{remark}
  To recover $\KhCx(L)$ from $\KhCx(L,\p)$, we can simply take the
  kernel of all the $y_p$ actions, and shift gradings by
  $(-m,-2m)$. That is,
  \[
  \KhCx(L)\cong \Sigma^{(-m,-2m)}\big(\bigcap_{p\in\p}\ker(y_p)\big).
  \]
  Alternatively, it is cokernel of all the $y_p$ actions:
  \[
  \KhCx(L)\cong \KhCx(L,\p)/\big(\sum_{p\in\p}\im(y_p)\big).
  \]
\end{remark}

\begin{remark} \label{rmk: roberts}
The Roberts--Jaeger twisted Khovanov complex \cite{RobertsTwisted, JaegerTwisted} is closely related to the above mapping cone construction. Roberts and Jaeger work with the reduced complex with coefficients in $\Z_2$. If we carry out their construction in the unreduced case with formal variables, then (up to a bigrading shift) their complex would be
\[
\Z_2[Y_{p_1},\dots,Y_{p_m}] \otimes \KhCx(L;\Z_2)
\]
with differential given by
\[
a\otimes b\mapsto a\otimes \Khdiff(b)+\sum_{p\in\p}Y_pa\otimes\xi_p(b).
\]
In particular, our $\KhCx(L,\p;\Z_2)$ is obtained from the Roberts--Jaeger construction by setting $Y_{p_i}^2=0$ for each $i$.

In~\cite{ManionSign}, Manion gave a sign-refinement of the Roberts--Jaeger complex to work over $\Z$; the underlying group is
\[
\Z[Y_{p_1},\dots,Y_{p_m}] \otimes \KhCx'(L;\Z).
\]
Manion's construction requires using the odd Khovanov complex; his arguments do not go through for the even version. One could again set $Y_{p_i}^2= 0$ to recover an odd analogue $\Kh'(L,\p)$ of our construction. However, as noted in Remark \ref{rmk: odd}, one should not expect this complex to be related to knot Floer homology. Indeed, as we shall see, the fact that $\Kh(L,\p)$ is a module over the exterior algebra $\Lambda_\p$ rather than $\Z[Y_{p_1},\dots,Y_{p_m}]/(Y_{p_1}^2,\dots,Y_{p_m}^2)$ is essential for seeing the similarity with knot Floer homology, specifically statement  \ref{item: Kh-properties: unlink} of Theorem \ref{thm: Kh-properties}.
\end{remark}

\subsection{Invariance and basic properties} \label{subsec: kh-cube-invariance}
In this section, we will prove the invariance of $\KhCx(L,\p)$, and
describe a few of its properties. Together, these results prove statements \ref{item: Kh-properties: invariance}, \ref{item: Kh-properties: doubling}, and \ref{item: Kh-properties: reduced} of Theorem \ref{thm: Kh-properties}.

First, we show that $\Kh(L,\p)$ is indeed a pointed link invariant.

\begin{proposition}\label{prop:main-invariance}
  The chain homotopy type of $\KhCx(L,\p)$ over $\Lambda_\p$ is an
  invariant of the isotopy type of $(L,\p)$.
\end{proposition}

\begin{proof}
  We have to check invariance under the three Reidemeister moves, and
  an additional move for moving a marked point past a crossing.\footnote{In light of Thomas Jaeger's clever handling of marked points in \cite{JaegerTwisted}, the first author refers to the four moves together as the ``Jaegermeister moves.''} Let us start with the basepoint move.

  Assume $(L,\p)$ and $(L,\p')$ are two pointed link diagrams
  differing only at a single basepoint on some link component; specifically, suppose $\p = \p'' \cup \{p\}$ and $\p' = \p'' \cup \{p'\}$. Since $p$ and $p'$ lie on the
  same component, by Lemma~\ref{lemma:movepoint} the actions $\xi_p$
  and $\xi_{p'}$ are chain homotopic via a homotopy that commutes with
  $\xi_q$ for all $q$; specifically, assume
  $H\from\KhCx(L)\to\KhCx(L)$ is the homotopy satisfying
  \[
  \Khdiff H+H\Khdiff=\xi_{p'}-\xi_{p}.
  \]

  For $a\in\Lambda_\p$, let $a'\in\Lambda_{\p'}$ be the corresponding
  element obtained by changing all occurrences of $y_p$ to $y_{p'}$.
  Then define the chain isomorphism $f\from\KhCx(L,\p)\to\KhCx(L,\p')$ to be
  \[
  f(a\otimes b)=a'\otimes b+(a'\wedge y_{p'})\otimes H(b)
  \]
  It is easy to see that $f$ commutes
  with the action of $\Lambda_{\p''}$, and intertwines the actions by
  $y_p$ and $y_{p'}$. It is also easy to check that $f$ is indeed a
  chain map:
  \begin{align*}
  &\phantom{==}f\Khdiff(a\otimes b)-\Khdiff f(a\otimes b)\\
  &=f\big((-1)^{\homgr(a)}a\otimes\Khdiff(b)+\sum_{q\in\p}(y_q\wedge
  a)\otimes \xi_q(b)\big)-\Khdiff\big(a'\otimes b+(a'\wedge y_{p'})\otimes H(b)\big)\\
  &=\big((-1)^{\homgr(a)}a'\otimes\Khdiff(b)+(-1)^{\homgr(a)}(a'\wedge
  y_{p'})\otimes H\Khdiff(b)+\sum_{q\in\p}(y'_q\wedge
  a')\otimes\xi_{q}(b)\\
  &\qquad{}+\sum_{q\in\p}(y'_q\wedge a'\wedge y_{p'})\otimes
  H\xi_q(b)\big)-\big(
  (-1)^{\homgr(a')}a'\otimes\Khdiff(b)+\sum_{q\in\p'}(y_q\wedge
  a')\otimes\xi_{q}(b)\\
&\qquad{}+(-1)^{\homgr(a')+1}(a'\wedge y_{p'})\otimes\Khdiff
H(b)+\sum_{q\in\p'}(y_q\wedge a'\wedge y_{p'})\otimes\xi_q H(b)
  \big)\\
  &=(y_{p'}\wedge a')\otimes (H\Khdiff+\Khdiff H)(b)+(y_{p'}\wedge
  a')\otimes\xi_{p}(b)-(y_{p'}\wedge a')\otimes\xi_{p'}(b)=0.
  \end{align*}



  Now we only have to check invariance under the Reidemeister
  moves. Thanks to the basepoint move, we may assume that the
  Reidemeister moves happen away from the basepoints; that is, we may
  assume that the two pointed link diagrams $(L,\p)$ and $(L',\p)$
  agree outside some disk $\mathbb{D}^2\subset\R^2$, differ by some
  Reidemeister moves inside $\mathbb{D}^2$, and the disk contains none
  of the basepoints of $\p$.

By Lemma~\ref{lemma:dp-action-away-from-basepoints}, Khovanov's chain
homotopy equivalence between $\KhCx(L)$ and $\KhCx(L')$ respects the
$D_{\p}$--action, and hence induces a chain homotopy equivalence between
$\KhCx(L,\p)$ and $\KhCx(L',\p)$ over $\Lambda_{\p}$.
\end{proof}

Next we study the dependence of $\KhCx(L,\p)$ on the basepoints $\p$.
\begin{lemma} \label{lemma: Kh-samecomponent} Let $(L,\p)$ be a
  pointed link diagram, and suppose that $\p$ contains two points
  $p_0, p_1$ that are on the same component of $L$. Let \[\p' = \p
  \minus \{p_0\}\,\,\textrm{ and }\,\,\p'' = \p
  \minus \{p_0,p_1\}.\] Then $\KhCx(L,\p)$ is isomorphic to
  $V\otimes\KhCx(L,\p')$, where $V$ is a direct sum of two copies of
  $\Z$, in bigradings $(0,0)$ and $(1,2)$. Furthermore, the
  isomorphism holds over $\Lambda_{\p'}$: it acts by restriction of
  scalars on $\KhCx(L,\p)$; and on $V$, $\Lambda_{\p''}$ acts
  trivially, and $y_{p_1}$ sends the copy of $\Z$ in grading $(0,0)$ to
  the copy of $\Z$ in grading $(1,2)$ by the identity map. In
  particular, over any field $\F$,
  \[
  \rank\Kh^\delta(L,\p;\F)=2\rank\Kh^\delta(L,\p';\F).
  \]
\end{lemma}

\begin{proof}
  By Proposition~\ref{prop:main-invariance}, we can assume that $p_0$
  and $p_1$ are adjacent to each other on the link diagram $L$. We have $\xi_{p_0}=\xi_{p_1}$ on
  $\KhCx(L)$, and therefore, they induce the same map, still denoted
  $\xi_{p_0}$ and $\xi_{p_1}$, on $\KhCx(L,\p'')$.  Then (up to some
  signs and bigrading shifts that we will suppress), $\KhCx(L,\p)$ is
  represented by the following mapping cone:
  \[
  \xymatrix{
    \KhCx(L,\p'')\ar[d]_-{\xi_{p_0}}\ar[r]^-{\xi_{p_1}}&    \KhCx(L,\p'')\ar[d]^-{-\xi_{p_0}}\\
    \KhCx(L,\p'')\ar[r]_-{\xi_{p_1}}&    \KhCx(L,\p'').
  }
  \]
  This is isomorphic to
  \[
  \left\{\vcenter{\hbox{\xymatrix{
    \KhCx(L,\p'')\ar[d]_-{\xi_{p_0}-\xi_{p_1}}\ar[r]^-{\xi_{p_1}}&    \KhCx(L,\p'')\ar[d]^-{\xi_{p_1}-\xi_{p_0}}\\
    \KhCx(L,\p'')\ar[r]_-{\xi_{p_1}}&    \KhCx(L,\p'')
  }}}\right\}
  =
  \left\{\vcenter{\hbox{\xymatrix{
    \KhCx(L,\p'')\ar[r]^-{\xi_{p_1}}&    \KhCx(L,\p'')\\
    \KhCx(L,\p'')\ar[r]_-{\xi_{p_1}}&    \KhCx(L,\p'')
  }}}\right\},
  \]
  with the isomorphism given by
  \[
  \left(\vcenter{\hbox{\xymatrix@C=0ex@R=0ex@M=0.2ex{
        a&\oplus&b\\
        \oplus&&\oplus\\
        c&\oplus&d
        }}}\right)
  \mapsto
  \left(\vcenter{\hbox{\xymatrix@C=0ex@R=0ex@M=0.2ex{
        a&\oplus&b\\
        \oplus&&\oplus\\
        (c-b)&\oplus&d
        }}}\right).
  \]
  It is immediate that the isomorphisms preserves the action by
  $\Lambda_{\p''}$, and it is not much harder to check that it
  respects the action by $y_{p_1}$ as well.
\end{proof}

When $p_0\in\p$ does not necessarily lie on the same component as some
other basepoint, we can only say the following:
\begin{lemma}\label{lemma:remove-one-basepoint}
  Let $(L,\p)$ be a pointed link diagram, and suppose that $\p$
  contains some point $p_0$, and $\p' = \p \minus \{p_0\}$.
  Then there is a short exact sequence
  \[
  0\to\Sigma^{(1,2)}\KhCx(L,\p')\to\KhCx(L,\p)\to\KhCx(L,\p')\to 0
  \]
  over $\Lambda_{\p'}$, which acts by restriction of scalars on
  $\KhCx(L,\p)$.  In particular, over any field $\F$,
  \[
  \rank\Kh^{\delta} (L,\p;\F)\leq 2\rank\Kh^{\delta}(L,\p';\F).
  \]
\end{lemma}

\begin{proof}
This is immediate from the mapping cone formulation. As before, if
$\xi_{p_0}$ denote the induced map on $\KhCx(L,\p')$, then
$\KhCx(L,\p)$ is given by the mapping cone
\[
\xymatrix{
\KhCx(L,\p')\ar[r]^-{\xi_{p_0}}&\KhCx(L,\p').
}
\]
Keeping track of the bigradings, the statement follows.
\end{proof}

Finally, we get some relation with the reduced Khovanov complex.
\begin{lemma} \label{lemma:Kh-reduced} Let $(L,\p)$ be a pointed link
  diagram, where $\p$ consists of a single point $p$. Then there is
  a short exact sequence
  \[
  0\to\Sigma^{(0,-1)}\rKhCx(L,p)\to\KhCx(L,\p)\to C \to 0,
  \]
  where $C$ is a chain complex that is chain homotopy equivalent to $\Sigma^{(1,3)}\rKhCx(L,p)$. Furthermore, over any field $\F$, the
  short exact sequence splits. Therefore, we have
  \[
  \rank\Kh(L,\p;\F) = 2\rank\rKh(L,p;\F).
  \]
  Keeping track of delta gradings,
  \[
  \rank\Kh^\delta(L,\p;\F) = \rank\rKh{}^{\delta+1/2} (L,p;\F) + \rank\rKh{}^{\delta-1/2} (L,p;\F).
  \]
\end{lemma}


\begin{proof}
  Let $C_1$, respectively $C_x$, denote the subgroup of $\KhCx(L)$
  where the pointed circle is labeled by $1$, respectively $x$. Then
  the Khovanov differential $\Khdiff$ restricts to a differential
  $d_x$ on $C_x$, making it a subcomplex of $\KhCx(L)$; and $C_1$ can
  be identified with the quotient complex $\KhCx(L)/C_x$, thus
  inheriting a quotient differential $d_1$. Furthermore, each of
  $\Sigma^{(0,1)}C_x$ and $\Sigma^{(0,-1)}C_1$ can be identified with
  the reduced complex $\rKhCx(L,p)$, and the Khovanov complex is the
  mapping cone of some chain map $f\from C_1\to C_x$, where $f$ is (up
  to some signs) the part of the Khovanov differential $\Khdiff$ that
  goes from $C_1$ to $C_x$.

  Up to some grading shifts and some global sign changes, the
  complex $\KhCx(L,\p)$ is given by the mapping cone
  \[
  \xymatrix{
    C_1\ar[d]_-f\ar[dr]|-X&C_1\ar[d]^-f\\
    C_x&C_x,
  }
  \]
  where the diagonal arrow $X$ is the identification between $C_1$ and
  $C_x$ coming from relabeling the marked circle from $1$ to $x$.

  Consider the inclusion $C_x\into\KhCx(L,\p)$ as the copy on the
  lower-left corner. The corresponding quotient complex is isomorphic
  to the mapping cone
  \[
  \xymatrix{
    C_1\ar[dr]|-X&C_1\ar[d]^-f\\
    &C_x,
  }
  \]
  which is chain homotopy equivalent to $C_1$ because $X\from C_1\to
  C_x$ is an isomorphism. Since, each of $C_x$ and $C_1$ are
  identified with the reduced complex $\rKhCx(L,p)$, the first
  statement follows (provided one kept track of the bigrading shifts).


  Over any field $\F$, Khovanov \cite[p.~189]{Kho-kh-Frobenius} showed that $\KhCx(L;\F)$ is isomorphic as a module over
  $D_{\p}\otimes\F=\F[\xi_p]/\xi_p^2$ to a direct sum of copies (with some grading shifts) of the   following three pieces:
  \[
   \left(
   \vcenter{\hbox{\xymatrix{
    a\\
    aX}}}
  \right),
  \qquad
  \left(
   \vcenter{\hbox{\xymatrix{
    a\ar[r]&b\\
    aX\ar[r]&bX}}}
  \right),
  \qquad
   \left(
   \vcenter{\hbox{\xymatrix{
    a\ar[dr]&b\\
    aX&bX}}}
  \right).
  \]
  (Actually, Khovanov only states this over $\Q$, but the argument works over any field $\F$.) Then, one can easily show that in each
  of these three cases, the inclusion $(C_x;\F)\into\KhCx(L,\p;\F)$
  splits. For instance, in the last case the the splitting map $\xymatrix{
  (C_x;\F)&\KhCx(L,\p;\F)\ar@{-->}[l]}$ is given by
  \[
  \vcenter{\hbox{\xymatrix{
    &&&&a\ar[dr]\ar[drr]&b\ar[drr]&a\ar[dr]&b\\
    aX&bX&&&aX\ar@{-->}@/_2ex/[llll]&bX\ar@{-->}@/_2ex/[llll]&aX\ar@{-->}@/^2ex/[lllll]^-{-1}&bX
  }}}.
  \]
  The other two cases are easier.
\end{proof}


Combining Lemmas~\ref{lemma:remove-one-basepoint}
and~\ref{lemma:Kh-reduced}, we obtain:
\begin{proposition} \label{prop: Kh-reduced2}
If $(L,\p)$ is a pointed link diagram, then over any field $\F$, and for any point $p\in\p$,
\[
\rank\Kh^\delta(L,\p;\F) \leq 2^{\abs{\p}-1} \left( \rank\rKh{}^{\delta+1/2} (L,p;\F) + \rank\rKh{}^{\delta-1/2} (L,p;\F)\right)
\]
When $L$ is a knot, this relation is an equality.
\end{proposition}

\subsection{Additional algebraic structure} \label{subsec: additional}

We now describe some additional algebraic structure on
$\KhCx(L,\p)$. Let $\Gamma_\p=\Lambda^*(\zeta_p\mid p\in\p)$ be the
exterior algebra on generators $\{\zeta_p\mid p\in\p\}$. The
$(\homgr,\intgr)$ of each $\zeta_p$ is $(-1,-4)$; define the bigrading
on $\Gamma_\p$ by extending it multiplicatively. The complex
$\KhCx(L,\p)$ carries a left $\Gamma_\p$ action by defining
\begin{equation} \label{eq: zetap}
\zeta_p(a \otimes b) = y_p^*(a) \otimes \xi_p(b),
\end{equation}
where $y_p^*$ denotes contraction by the dual element to $y_p$. To see that this is a valid action, observe:
\begin{align*}
&\phantom{==}\zeta_p(\Khdiff(a \otimes b))+\Khdiff(\zeta_p(a
\otimes b))\\
 &= \big((-1)^{\homgr(a)} y_p^*(a) \otimes
\xi_p(\Khdiff(b)) +\sum_{q \in \p} y_p^*(y_q\wedge a)
\otimes \xi_p(\xi_q(b))) \\
&\qquad\qquad +\big((-1)^{\homgr(a)-1} y_p^*(a) \otimes
\Khdiff(\xi_p(b)) + \sum_{q \in \p} (y_q\wedge y_p^*(a)) \otimes \xi_q(\xi_p(b))\big)\\
&=\sum_{q\in\p}(y_p^*(y_q\wedge a)+y_q\wedge
y_p^*(a))\otimes \xi_q(\xi_p(b))\\
&=a\otimes\xi_p^2(b)=0.
\end{align*}

To see the interaction of the left $\Gamma_\p$--module structure and
the left $\Lambda_\p$--module structure of $\KhCx(L,\p)$, we have:
\begin{align*}
&\phantom{==}\zeta_p(y_q(a \otimes b))+y_q(\zeta_p(a \otimes b))\\
 &= y_p^*((y_q\wedge a)) \otimes \xi_p(b)+(y_q\wedge y_p^*(a)) \otimes \xi_p(b) \\
&= \delta_{p,q} a \otimes \xi_p(b),
\end{align*}
where $\delta_{p,q}$ denotes Kronecker delta. Thus, the actions by $\zeta_p$ and $y_q$ anticommute on the nose when $p \ne q$, and up to
chain homotopy when $p=q$; the homotopy in the latter case is given by
$H_p(a \otimes b) = y_p^*(a)\otimes b$ since:
\begin{align*}
&\phantom{==} H_p(\Khdiff(a \otimes b)) +\Khdiff(H_p(a\otimes b))\\
 &= \big((-1)^{\homgr(a)} y_p^*(a) \otimes \Khdiff(b) + \sum_{q \in
   \p} y_p^*(y_q \wedge a)\otimes\xi_q(b)\big)\\
 &\qquad\qquad +\big((-1)^{\homgr(a)-1} y_p^*(a) \otimes \Khdiff(b) + \sum_{q \in \p} (y_q \wedge y_p^*(a))\otimes\xi_q(a)\big)\\
&= \delta_{p,q}a\otimes\xi_q(b).
\end{align*}

\begin{remark}
  This computation suggests that there is a left action on
  $\KhCx(L,\p)$ by some non-commutative dg-algebra generated by $y_p$,
  $\zeta_p$, and $H_p$, containing $\Lambda_\p$ and $\Gamma_\p$ as
  sub-algebras, and with $y_p$ and $\zeta_q$ (graded) commuting for
  all distinct $p,q$, and $\bdy H_p=y_p\zeta_p+\zeta_py_p$.  We will
  not concern ourselves with questions about invariance of these new
  maps $\zeta_p$, and the possible existence of higher
  $A_\infty$-structures involving $y_p$ and $\zeta_p$, although the
  relevance of these maps will soon become apparent.
\end{remark}

\subsection{Planar unlinks} \label{sec: Kh-unlink}

In this section, we consider the special case of a diagram with zero crossings. Let $(L,\p)$ be a non-degenerate, $k$-component planar unlink, oriented as the boundary of the black region.  Let $[L]=\{L_1, \dots, L_k\}$ denote the set of components of $L$, and let $\pi \co \p \to [L]$ be the natural surjection induced by the inclusion of $\p$.

For each $i\in\{1, \dots, k\}$, define $\alpha_{i} = \sum_{p \in \p_i} y_p \in \Lambda_{\p}$. Let $\Lambda_{\p,L}$ denote the quotient of $\Lambda_\p$ by the ideal generated by $\alpha_1, \dots, \alpha_k$.

Let $\Gamma_L$ denote the exterior algebra on generators $\gamma_1, \dots, \gamma_k$ corresponding to the components of $L$, with each generator in bigrading $(-1,-4)$. Given a section $s$ of $\pi$ (i.e. a choice of $p_{s(i)} \in \p \cap L_i$ for each $i$), the inclusion $\gamma_i \mapsto \zeta_{p_{s(i)}}$ makes $\Gamma_L$ a subring of $\Gamma_\p$.

By definition, as groups, $\KhCx(L,\p) \cong \Lambda_\p\otimes
\KhCx(L)$. After shifting bigradings by $(0,k)$, we may identify this
with the algebra generated by
\[
\{y_p \mid p \in \p\} \cup \{x_i\mid L_i\in[L]\},
\]
with each $y_p$ in bigrading $(1,2)$ and each $x_i$ in bigrading $(0,-2)$, with the following relations:
\begin{equation} \label{eq: Kh-unlink-relations}
x_i^2 = y_p^2 = 0, \qquad  x_i x_j = x_j x_i, \qquad y_p x_i = x_i y_p, \qquad y_p y_{q} = -y_{q} y_p.
\end{equation}
The differential is then left multiplication by the element
$\sum_{i=1}^k \alpha_{i} x_i$. The left action by $y_p$ is given by
left multiplication by $y_p$, and the left action by $\zeta_p$ is
given by \( x_{\pi(p)} y_p^*(\cdot), \) where $L_{\pi(p)}\in[L]$
is the component containing $p$.

\begin{lemma}\label{lem:action-on-homology}
  The left actions by $\Lambda_\p$ and $\Gamma_\p$ on $\KhCx(L,\p)$
  descend to a well-defined (that is, independent of the choice of
  section $s$ of $\pi$) left action by the graded tensor product
  $\Lambda_{\p,L} \otimes \Gamma_L$ on homology.
  $\Kh(L,\p)$. Specifically:
\begin{enumerate}[leftmargin=*]
\item\label{lem:action-on-homology:alpha} For any $i\in\{1,\dots,k\}$,
  the action $(-1)^{\beta_i(\cdot)}\alpha_{i}(\cdot)$ on
  $\KhCx(L,\p)$ is nullhomotopic, where for any $w\in\KhCx(L,\p)$,
  $\beta_i(w)$ is defined to be
  \[
  \beta_i(w)=\begin{cases}
    0&\text{if $x_i$ does not appear in $w$,}\\
    1&\text{if $x_i$ appears in $w$.}
  \end{cases}
  \]
\item\label{lem:action-on-homology:y-zeta} For any $p,q\in\p$, the
  actions by $y_p\zeta_q$ and $-\zeta_q y_p$ on $\KhCx(L,\p)$ are
  chain homotopic.
\item\label{lem:action-on-homology:zeta} For any $i\in\{1,\dots,k\}$
  and for any $p,q\in\p_i$, the actions by $\zeta_p$ and $\zeta_q$ on
  $\KhCx(L,\p)$ are chain homotopic.
\end{enumerate}
\end{lemma}

\begin{proof}
For Part~(\ref{lem:action-on-homology:alpha}), define the homotopy to be \(H(w)=x^*_i(w).\) Following the sign convention in \eqref{eq: homotopy-sign}, we have:
\begin{align*}
&\phantom{==}H\Khdiff(w)-\Khdiff H(w)\\
&=x^*_i(\sum_j \alpha_{j} x_jw)-\sum_j \alpha_{j} x_jx_i^*(w)\\
&=\alpha_{i}(x_i^*x_i-x_ix_i^*)(w)=(-1)^{\beta_i(w)}\alpha_{i}(w).
\end{align*}

Part~(\ref{lem:action-on-homology:y-zeta}) is proved in
Section~\ref{subsec: additional}.

For Part~(\ref{lem:action-on-homology:zeta}), define the (graded) null-homotopy to be $H(w)=y^*_p y^*_q(w)$. Then
\begin{align*}
&\phantom{==}H\Khdiff(w)-\Khdiff H(w)\\
&=y^*_py^*_q(\sum_j\alpha_{j}x_jw)-\sum_j\alpha_{j}x_jy^*_py^*_q(w)\\
&=\sum_j
x_j\big(y^*_p(y^*_q\alpha_{j}+\alpha_{j}y^*_q)-(y^*_p\alpha_{j}+\alpha_{j}
y^*_p)y^*_q\big)(w)\\
&=x_i(y_p^*(w)-y_q^*(w))\\
&=\zeta_p(w)-\zeta_q(w).\qedhere
\end{align*}
\end{proof}

\begin{proposition} \label{prop: Kh-unlink} If $(L, \p)$ is a
  non-degenerate, $k$-component pointed unlink (see
  Definition~\ref{def:pointed-link}), the homology $\Kh(L,\p)$ is
  isomorphic as a
  $\Lambda_{\p,L}\otimes\Gamma_L$--module to
  $\Lambda_{\p,L}\otimes\Gamma_L$, generated by the
  homology class of $\alpha_{1} \cdots \alpha_{k}$. (Note that
  the choice of this generator depends up to sign on a choice of ordering
  of the components of $L$, modulo alternating permutations.)
\end{proposition}

\begin{proof}
There is a natural isomorphism of chain complexes
\[
\KhCx(L,\p) \cong \KhCx(L_1, \p_1) \otimes \cdots \otimes \KhCx(L_k,\p_k),
\]
and the actions of $\Lambda_{\p,L}$ and $\Gamma_{L}$
respect this decomposition. By the K\"unneth formula, it thus suffices
to consider the components separately. Each complex $\KhCx(L_i,\p_i)$
is a mapping cone
\[
\Lambda_{\p_i} \xrightarrow{\alpha_{i}x_i} \Lambda_{\p_i}x_i,
\]
whose homology is evidently $\alpha_{i} \Lambda_{\p_i} \oplus
\left(\Lambda_{\p_i} /\alpha_{i}\right)x_i$; each summand is
isomorphic to $ \Lambda_{\p_i,L_i}$. Moreover, we have
$\zeta_{p_{s(i)}}(\alpha_{i}) = x_i$, where $p_{s(i)}\in\p_i$ is the
chosen basepoint in the component $L_i$. Thus,
\[
\Kh(L,\p) \cong \Kh(L_1, \p_1) \otimes \cdots \otimes \Kh(L_k,\p_k),
\]
and it is generated as a
$\Lambda_{\p,L} \otimes \Gamma_{L}$--module by
the homology class of $\alpha_{1} \otimes \cdots \otimes
\alpha_{k}$.
\end{proof}

\subsection{Elementary cobordisms for planar unlinks} \label{subsec: Kh-cobordism}
In this subsection, we will study planar unlinks that are related by
elementary saddle cobordisms. More specifically, assume $L$ is a $k$-component
planar unlink as before, and a new planar unlink $L'$ is obtained from
$L$ by splitting a component into two components. The original Khovanov split and merge maps $\Delta$ and $\mu$ (coming from the Frobenius comultiplication and multiplication, respectively) extend to maps
\begin{align*}
\tilde\Delta &= \id_{\Lambda_\p}\otimes \Delta\co\KhCx(L,\p) \to \KhCx(L',\p) \\
\tilde\mu &= \id_{\Lambda_\p} \otimes\mu \co \KhCx(L',\p) \to \KhCx(L,\p).
\end{align*}
Since the $\xi_p$--actions commute with the split and merge maps, we
see that $\tilde\Delta$ and $\tilde\mu$ are both chain maps. Our goal is to determine the induced maps on homology,
\begin{align*}
\tilde\Delta_* \co & \Kh(L,\p) \to \Kh(L',\p) \\
\tilde\mu_* \co & \Kh(L',\p) \to \Kh(L,\p),
\end{align*}
in terms of the module structures given by Proposition \ref{prop: Kh-unlink}.

Assume the components of $L$ and $L'$ are each totally ordered; we do not require the orderings to be compatible in any way. To begin, we will define maps
\[
\xymatrix{\Lambda_{L,\p}\ar@/^1ex/[r]^-{\Delta_{\Lambda}}&\Lambda_{L',\p}\ar@/^1ex/[l]^-{\mu_{\Lambda}}}\qquad\xymatrix{\Gamma_{L}\ar@/^1ex/[r]^-{\Delta_{\Gamma}}&\Gamma_{L'}\ar@/^1ex/[l]^-{\mu_{\Gamma}}}.
\]
that depend (up to sign) on the choices of ordering.

First let us assume that the ordered components of $L$ are
$L_1,L_2\dots,L_k$, and the ordered components of $L'$ are
$L'_0,L'_1,L_2,\dots,L_k$, so that $L_1$ splits into components $L'_0$
and $L'_{1}$, where the split occurs away from $\p$.  Define
$\alpha'_i = \sum_{p \in L'_i} y_p$ for $i=0,1$, so that $\alpha_1 =
\alpha'_0 + \alpha'_1$. For convenience, write $L'_i = L_i$ for $i=2, \dots, k$ and $\p'_i = L'_i \cap \p$ for $i=0, \dots, k$.

Let us denote the algebra generators of $\Gamma_{L}$
by $\gamma_1, \dots, \gamma_k$, and those of $\Gamma_{L'}$ by
$\gamma'_0, \dots, \gamma'_{k}$; that is, $\gamma_i$ acts on
$\KhCx(L,\p)$ as $\zeta_{p_{s(i)}}$ for some $p_{s(i)}\in\p_i$ and
$\gamma'_i$ acts on $\KhCx(L',\p)$ as $\zeta_{p'_{s'(i)}}$ for some
$p'_{s'(i)}\in\p'_i$. We may naturally identify $\Gamma_{L}$ as
the quotient $\Gamma_{L'} / (\gamma'_0 - \gamma'_{1})$; let
\[
\mu_\Gamma \co \Gamma_{L'} \to \Gamma_{L}
\]
be the quotient map. There is then an inclusion
\[
\Delta_\Gamma \co \Gamma_{L} \to \Gamma_{L'}
\]
taking $1 \mapsto \gamma'_{0} - \gamma'_{1}$.

Likewise, since
\[
 \Lambda_{\p,L} = \Lambda_\p/(\alpha'_0 + \alpha'_1, \alpha_2,\dots,\alpha_k)
\]
and
\[
 \Lambda_{\p,L'} = \Lambda_\p/(\alpha'_0,\alpha'_1,\alpha_2,\dots,\alpha_k),
\]
let
\[
\Delta_\Lambda \co  \Lambda_{\p,L} \to  \Lambda_{\p,L'}
\]
denote the obvious quotient map. There is then an inclusion
\[
\mu_\Lambda \co  \Lambda_{\p,L'} \to \Lambda_{\p,L}
\]
taking $1 \in  \Lambda_{\p,L'}$ to $\alpha'_0 \in
\Lambda_{\p,L}$, which equals $-\alpha'_1$.


Note that $\mu_\Gamma$ and $\Delta_\Lambda$ do not depend on the
ordering of the components of $L$ and $L'$, while $\Delta_\Gamma$ and
$\mu_\Lambda$ have only been defined for the above specific
ordering. Now instead, if we start with some arbitrary ordering of the
components of $L$ and $L'$, we may change the orderings to the above
specific ordering by some permutations $\sigma\in S_k$ and $\tau\in
S_{k+1}$, and then we modify the maps $\Delta_\Gamma$ and $\mu_\Lambda$ by
multiplying by $\mathrm{sgn}(\sigma)\mathrm{sgn}(\tau)$, where
$\mathrm{sgn}$ is the permutation signature.

We are now prepared to describe $\tilde \Delta_*$ and $\tilde \mu_*$.

\begin{lemma}\label{lem:split-merge-module}
  The maps $\tilde\Delta_*$ and $\tilde{\mu}_*$ are
  $\Lambda_{\p,L}\otimes\Gamma_{L'}$--module maps, where
  $\Lambda_{\p,L}$ acts on $\Kh(L',\p)$ by extension of scalars by
  $\Delta_{\Lambda}$, and $\Gamma_{L'}$ acts on $\Kh(L,\p)$
  by extension of scalars by $\mu_{\Gamma}$. That is, for any
  $\lambda\in\Lambda_{\p,L}$, $\gamma\in\Gamma_{L'}$,
  $x\in\Kh(L,\p)$, and $y\in\Kh(L',\p)$,
  \begin{align*}
    \tilde\Delta_*((\lambda\otimes\mu_{\Gamma}(\gamma))x)&=(\Delta_{\Lambda}(\lambda)\otimes\gamma)(\tilde\Delta_*(x))\\
    \tilde\mu_*((\Delta_{\Lambda}(\lambda)\otimes\gamma)y)&=(\lambda\otimes\mu_{\Gamma}(\gamma))(\tilde\mu_*(y)).
  \end{align*}
\end{lemma}

\begin{proof}
  It is not hard to check that the actions of $y_p$ and $\zeta_p$ on
  $\KhCx(L,\p)$ and $\KhCx(L',\p)$ commute with $\tilde \Delta$ and
  $\tilde \mu$; the result follows.
\end{proof}

\begin{proposition} \label{prop: Kh-edgemaps} Let $(L,\p)$ and
  $(L',\p)$ be pointed planar unlinks as described above, and assume
  that both are non-degenerate. Choose module identifications
  $\Kh(L,\p) \cong \Lambda_{\p,L}\otimes\Gamma_{L}$ and
  $\Kh(L',\p) \cong \Lambda_{\p,L'}\otimes\Gamma_{L'}$ as
  in Proposition~\ref{prop: Kh-unlink}.  Then $\tilde \Delta_*$ and
  $\tilde \mu_*$ are given by
\begin{align}
\label{eq: Kh-split}
\tilde\Delta_* &= \Delta_\Lambda\otimes\Delta_\Gamma  \\
\label{eq: Kh-merge}
\tilde\mu_* &= \mu_\Lambda\otimes\mu_\Gamma.
\end{align}
\end{proposition}



\begin{proof}
  We are free to choose the ordering of the components of $L$ and $L'$
  since if we modify the ordering by permutations $\sigma\in S_k$ and
  $\tau\in S_{k+1}$, then both sides of the equation get multiplied by
  $\mathrm{sgn}(\sigma)\mathrm{sgn}(\tau)$. So order the components of
  $L$ and $L'$ as before, so that $1$ denotes the homology class of
  $\alpha_1 \cdots \alpha_k$ in $\Kh(L,\p)$ and the homology
  class of $ \alpha'_0 \alpha'_1 \cdots \alpha'_k$ in
  $\Kh(L',\p)$. We have:
  \begin{align*}
\tilde \Delta(\alpha_1 \cdots \alpha_{k})
    &= \alpha_1 \cdots \alpha_{k}(x'_0 + x'_1)  \\
    &=  (\alpha'_0 + \alpha'_1) \alpha'_{2} \cdots \alpha'_{k}(x'_0 + x'_1)   \\
    &= (\zeta_{p'_{s'(0)}} - \zeta_{p'_{s'(1)}}) (\alpha'_0 \alpha'_1 \alpha'_{2} \cdots
    \alpha'_{k}) +  \Khdiff (\alpha'_2 \cdots \alpha'_k).
  \end{align*}
  Hence, on the level of homology,
  \[
  \tilde\Delta_*(1) = \big(1\otimes (\gamma'_0 - \gamma'_1)\big)(1) = (\Delta_\Lambda\otimes\Delta_\Gamma)(1),
\]
and since $\mu_{\Gamma}$ is surjective, by
Lemma~\ref{lem:split-merge-module} the general case follows.

Similarly, we have:
\begin{align*}
\tilde \mu(\alpha'_0\alpha'_1 \cdots\alpha'_{{k}})
 &= \alpha'_0\alpha'_1 \cdots\alpha'_{{k}}\\
&= \alpha'_0(\alpha_1 - \alpha'_0)\alpha_{2} \cdots \alpha_{{k}}   \\
&= \alpha'_0(\alpha_1 \cdots \alpha_{k}).
\end{align*}
Thus, on homology,
\[
\tilde\mu_*(1) = (\alpha'_0\otimes 1)(1) = (\mu_\Lambda \otimes \mu_{\Gamma})(1);
\]
and since $\Delta_{\Lambda}$ is surjective, by
Lemma~\ref{lem:split-merge-module} we are done.
\end{proof}


\subsection{The cubical filtration}\label{subsec: kh-cube-filtration}

We return to the case of a general pointed link diagram $(L,\p)$. Let
$\FF$ be the filtration on $\KhCx(L,\p)$ defined so that the summand
$\KhCx(L_v,\p)$ (where $v \in \{0,1\}^n$) is in filtration level
$\abs{v}$.

The associated graded complex can be identified with \( \bigoplus_{v
  \in \{0,1\}^n} \KhCx(L_v,\p) \). The two complexes agree up to signs
of the differentials. To wit, if $p\in\p$ lies on an edge of $L$ of
parity $\epsilon(p)$, the $\xi_p$--action on $\KhCx(L)$ comes with a
sign of $(-1)^{\epsilon(p)}$; on the other hand, we have oriented the
planar unlinks $L_v$ as the boundary of the black regions, so the
$\xi_p$--actions on $\KhCx(L_v)$ always come with a positive
sign. Nevertheless, we still have an isomorphism
\[
E_0(\KhCx(L,\p), \Khdiff, \FF) \cong \bigoplus_{v \in \{0,1\}^n} \KhCx(L_v,\p),
\]
given by
\[
a\otimes b\mapsto (-1)^{\epsilon(p_{i_1})+\dots+\epsilon(p_{i_k})}a\otimes b\qquad\text{ for $a=y_{p_{i_1}}\wedge\cdots\wedge y_{p_{i_k}}\in\Lambda_\p$, $b\in\KhCx(L)$.}
\]
By a slight abuse of notation, in this subsection, whenever we talk
about $\KhCx(L,\p)$, we mean the complex obtained by post-composing
with the above isomorphism.

The $d_0$ differential on each summand is the standard differential on
$\KhCx(L_v,\p)$. Therefore if each $L_v$ is non-degenerate, the
$E_1$ page is
\[
E_1(\KhCx(L,\p), \Khdiff, \FF) = \bigoplus_{v \in \{0,1\}^n} \Kh(L_v,\p)\cong \bigoplus_{v \in \{0,1\}^n} (\Lambda_{\p,L_v}\otimes\Gamma_{L_v}),
\]
where the second identification is via Proposition~\ref{prop:
  Kh-unlink}, and depends only the ordering of the components of $L_v$
for each $v$.

Let $\tilde{d}_1$ be the sum, taken over all pairs $u \iscovered v$,
of the maps
\[
\id_{\Lambda_{\p}} \otimes\Khdiff^{u,v} \co \KhCx(L_u,\p) \to \KhCx(L_{v},\p).
\]
Then the $d_1$ differential is the map induced on the homology of
$(E_0,d_0)$ by the map
\[
a\otimes b \mapsto (-1)^{\homgr(a)}\tilde{d}_1(a\otimes b)\qquad\text{ for $a\in\Lambda_\p$, $b\in\KhCx(L)$.}
\]
Since $\id_{\Lambda_{\p}} \otimes\Khdiff^{u,v}$ is either
$(-1)^{s_{u,v}+n_-}\tilde \Delta$ or $(-1)^{s_{u,v}+n_-}\tilde \mu$
according to whether the edge $u \iscovered v$ is a split or a merge,
the induced map on homology is given by Proposition~\ref{prop:
  Kh-edgemaps}. This completely determines the $d_1$
differential. Specifically, the component of the $d_1$ differential
from $(\Lambda_{\p,L_u}\otimes\Gamma_{L_u,\p})$ to
$(\Lambda_{\p,L_v}\otimes\Gamma_{L_v,\p})$ is given by:
\[
\lambda\otimes\gamma\mapsto(-1)^{s_{u,v}+n_-+\homgr(\lambda)+\homgr(\gamma)+\card{L_u}}\begin{cases}
\Delta_{\Lambda}(\lambda)\otimes\Delta_{\Gamma}(\tilde\gamma)\text{ for
  some $\tilde\gamma$ with $\mu_{\Gamma}(\tilde\gamma)=\gamma$}&\\
&\text{\llap{if $u\iscovered v$ is a split,}}\\
\mu_{\Lambda}(\tilde\lambda)\otimes\mu_{\Gamma}(\gamma)\text{ for
  some $\tilde\lambda$ with $\Delta_{\Lambda}(\tilde\lambda)=\lambda$}&\\
&\text{\llap{if $u\iscovered v$ is a merge.}}
\end{cases}
\]
(Note that we can represent the homology element
$\lambda\otimes\gamma\in\Lambda_{\p,L_u}\otimes\Gamma_{L_u,\p}$
by some cycle $\sum_ia_i\otimes b_i\in\Lambda_\p\otimes\KhCx(L_u)$
with $\homgr(a_i)=\homgr(\lambda)+\homgr(\gamma)+\card{L_u}$ for all $i$.)

To summarize:
\begin{proposition}
  Let $(L,\p)$ be a $k$-component pointed link diagram with $n$
  (ordered) crossings and $m$ (ordered) basepoints, so that for each
  $v\in\{0,1\}^n$, the complete resolution $L_v$ is non-degenerate (see
  Definition~\ref{def:pointed-link}). Then the
  $(\homgr,\intgr)$-graded complex $\KhCx(L,\p)$ (see
  Section~\ref{subsec: mapping-cone-def}) admits a cubical filtration
  by $\{0,1\}^n$ (see Definition~\ref{defn:cubical-chain-complex}).

  For the corresponding (bigraded) spectral sequence converging to
  $\Kh(L,\p)$, the $E_1$-page summand at a vertex $v\in\{0,1\}^n$ is
  isomorphic to
  \[
  \Lambda_{\p,L_v}\otimes\Gamma_{L_v}=\Lambda_{\p,L_v}\otimes\Gamma_{L_v},
  \]
  with the isomorphism only depending on the ordering of the
  components of $L_v$ (see Proposition~\ref{prop: Kh-unlink}). The
  $d_1$-differential along the edge $u\iscovered v$ agrees (up to some
  signs that depend on $u$, $v$, the bigradings, and the ordering of
  the components of $L_u$ and $L_v$) with the following map: For
  $\lambda\in\Lambda_{\p,L_u},\gamma\in\Gamma_{L_u}$,
  \[
  \lambda\otimes\gamma\mapsto\begin{cases}
\Delta_{\Lambda}(\lambda)\otimes\Delta_{\Gamma}(\tilde\gamma)&\text{for
  some $\tilde\gamma\in\Gamma_{L_v}$ with
  $\mu_{\Gamma}(\tilde\gamma)=\gamma$ if $u\iscovered v$ is a split,}\\
\mu_{\Lambda}(\tilde\lambda)\otimes\mu_{\Gamma}(\gamma)&\text{for
  some $\tilde\lambda\in\Lambda_{\p,L_v}$ with
  $\Delta_{\Lambda}(\tilde\lambda)=\lambda$ if $u\iscovered v$ is a merge,}
\end{cases}
\]
(see Proposition~\ref{prop: Kh-edgemaps}).
\end{proposition}

\section{Background on knot Floer homology} \label{sec: HFK}

In this section, we review a few facts about knot Floer homology. Unlike in many other papers, we work with coefficients in $\Z$ in order to illustrate certain connections with the Khovanov homology for pointed links discussed above.

\subsection{Orientation systems and integer coefficients}

Given a non-degenerate, oriented, pointed link $\LL = (L,\p)$ in $S^3$, consider the sutured $3$-manifold $Y(\LL) = S^3 \minus \nbd(L)$, with the sutures $\gamma_\p \subset \partial Y(\LL)$ defined to be a collection of $2\abs{\p}$ meridians of $L$, arranged such that $R_-(\gamma_\p)$ consists of a meridional annulus for each point of $\p$ and $R_+(\gamma_\p)$ consists of a meridional annulus for each component of $L \minus \p$. Let $\nextpt \co \p \to \p$ be the bijection that sends each point of $\p$ to the point that follows it as one traverses $L$ according to the orientation. Let us assume that $L$ has $l$ components $L_1, \dots, L_l$, and that $\p = \{p_1, \dots, p_m\}$.

Let $\HH = (\Sigma, \bm\alpha, \bm\beta, \w, \z)$ be a multi-pointed Heegaard diagram for the pointed link $\LL$. (We follow the notation in \cite[Section 3]{BaldwinLevineSpanning}, except that we use $\w$ and $\z$ in place of $\mathbb{O}$ and $\mathbb{X}$, respectively.) We may label the basepoints $\w = \{w_p \mid p \in \p\}$ and $\z = \{z_p \mid p \in \p\}$, so that $w_p$ and $z_p$ are contained in the same component of $\Sigma \minus \bm\alpha$, and $z_p$ and $w_{\nextpt(p)}$ are contained in the same component of $\Sigma \minus \bm\beta$. Let us label these components $A_1, \dots, A_m$ and $B_1, \dots, B_m$, such that $w_{p_i} \in A_i \cap B_i$. Observe that each $A_i$ or $B_i$ has genus $0$.

Let $\Pi_\HH$ denote the group of periodic domains, and let $\Pi_\HH^0$ denote the subgroup of periodic domains that avoid all the basepoints. It is not hard to see that $\Pi_\HH$ is generated by $A_1, \dots, A_m, B_1, \dots, B_m$, with the single relation that $A_1 + \dots + A_m = B_1 + \dots + B_m$. Moreover, $\Pi_\HH^0$ is $(l-1)$--dimensional, generated by periodic domains $P_i = \sum_{j \mid p_j \in \p \cap L_i} (A_j-B_j)$, which satisfy the relation $P_1 + \cdots + P_l = 0$.

As a notational convenience, we may record the basepoint multiplicities of each Whitney disk $\phi$ using formal linear combinations of the marked points $p$. That is, we define
\[
\mults_\w(\phi) = \sum_{p \in \p} n_{w_p}(\phi) \cdot p \quad \text{and} \quad
\mults_\z(\phi) = \sum_{p \in \p} n_{z_p}(\phi) \cdot p.
\]
We also write
\[
n_\w(\phi) = \sum_{p \in \p} n_{w_p}(\phi) \quad \text{and} \quad
n_\z(\phi) = \sum_{p \in \p} n_{z_p}(\phi) .
\]

Let $\CFKtil(\HH)$ denote the $\Z$--module generated freely by points of $\T_\alpha \cap \T_\beta$. This complex comes equipped with absolute Maslov and Alexander gradings, which are characterized up to overall shifts by, for any $\x,\y \in \T_\alpha \cap \T_\beta$ and any $\phi \in \pi_2(\x,\y)$, the formulas
\[
M(\x) - M(\y) = \mu(\phi) - 2n_\w(\phi) \quad \text{and} \quad A(\x) - A(\y) = n_\z(\phi) - n_\w(\phi)
\]
where $\mu(\phi)$ denotes the Maslov index of $\phi$. Throughout this section, the bigradings of elements and of maps will be given as (Maslov, Alexander). The delta grading is defined by
\[
\delta(\x) = A(\x) - M(\x).
\]

In order to work with coefficients in $\Z$, we must specify a \emph{system of orientations}: a choice, for each $\x,\y \in \T_\alpha \cap \T_\beta$ and each $\phi \in \pi_2(\x,\y)$, of a trivialization of the determinant line bundle $\det(\phi)$ associated to $\phi$, subject to suitable compatibility conditions. (We refer to \cite{Sar-signs} for definitions and basic results.) Such a choice determines orientations on all the moduli spaces $\hat\MM(\phi)$. Two orientation systems are called \emph{strongly equivalent} if they agree on all elements of $\Pi_\HH$, and \emph{weakly equivalent} if they agree on all elements of $\Pi_\HH^0$. Let $\OO_\HH$ (resp.~$\hat\OO_\HH$) denote the set of strong (resp.~weak) equivalence classes of orientation systems, which has order $2^{2m-1}$ (resp.~$2^{l-1}$).

The differential $\partial$ is defined by
\begin{equation} \label{eq: HFK-diff}
\partial(\x) = \sum_{\y \in \T_\alpha \cap \T_\beta} \sum_{ \substack{\phi \in \pi_2(\x, \y) \\ \mu(\phi)=1 \\ \mults_\w(\phi) = \mults_\z(\phi) = 0 }} \#\widehat{\MM}(\phi) \, \x,
\end{equation}
and satisfies $\partial^2=0$; moreover, up to isomorphism, this chain complex depends only on the weak equivalence class of the orientation system. Thus, there are $2^{l-1}$ (potentially) different homology groups, which we denote $\HFKtil(\HH;\Z,\ort)$, for $\ort \in \hat\OO_\HH$. The third author showed that the collection of these homology groups is an invariant of the pointed link $\LL$. To be precise, if $\HH$ and $\HH'$ are related by a sequence of isotopies, handleslides, and stabilizations, there is a bijection $f \co \hat\OO_\HH \to \hat\OO_{\HH'}$ such that
\[
\HFKtil(\HH; \Z,\ort) \cong \HFKtil(\HH';\Z,f(\ort))
\]
for each $\ort \in \hat\OO_\HH$ \cite[Theorem 2.4]{Sar-signs}.

For any $\x \in \T_\alpha \cap \T_\beta$, let $\pi_2^\alpha(\x)$ and $\pi_2^\beta(\x)$ denote the groups of $\alpha$-- and $\beta$--boundary degenerations based at $\x$. Each $\phi \in \pi_2^\alpha(\x)$ or $\pi_2^\beta(\x)$ has a moduli space of holomorphic representatives $\NN(\phi)$ with a codimension-$2$ quotient $\hat\NN(\phi)$. Alishahi and Eftekhary \cite[Lemma 5.2]{AlishahiEftekharySutured} show that there is an orientation system with the following property: For any $\phi \in \pi_2^\alpha(\x)$ (resp.~$\pi_2^\beta(\x)$), the \emph{signed} count $\#\hat\NN(\phi)$ is equal to $1$ if the domain of $\phi$ is equal to $A_i$ (resp.~$B_i$) for some $i$, and $0$ otherwise. (The analogous statement without signs was proven by Ozsv\'ath and Szab\'o \cite[Theorem 5.5] {OSzLink}.) Any two orientation systems having this property must have the same values on all periodic domains; in particular, they are strongly equivalent. We refer to this equivalence class of orientation systems as $\ort_{\can}(\HH)$. By tracing through the arguments from \cite{Sar-signs}, it is not hard to see that the bijection $f$ above takes $\ort_{\can}(\HH)$ to $\ort_{\can}(\HH')$. Thus, $\HFKtil(\HH;\Z,\ort_{\can}(\HH))$ gives a well defined invariant of $\LL$, which we will henceforth just denote by $\HFKtil(\LL;\Z)$ or simply $\HFKtil(\LL)$. When $\p$ contains exactly one point on each component of $L$, we write $\HFK(L;\Z)$ for $\HFKtil(\LL;\Z)$. We write the decomposition according to delta gradings as
\[
\HFK(L) = \bigoplus_{\delta \in \Z + \frac{l}{2}} \HFK{}^\delta(L) \quad \text{and} \quad
\HFKtil(\LL) = \bigoplus_{\delta \in \Z + \frac{l}{2}} \HFKtil{}^\delta(\LL). \]

A key property of $\ort_{\can}(\HH)$, proved by Alishahi and Eftekhary in the proof of \cite[Theorem 5.7]{AlishahiEftekharySutured}, is the following:

\begin{lemma} \label{lemma: boundary-degen}
Let $\x \in \T_\alpha \cap \T_\beta$, and let $\phi \in \pi_2^\alpha(\x) \cup \phi_2^\beta(\x)$ be a boundary degeneration whose domain is either $A_i$ or $B_i$. Then
\begin{equation} \label{eq: boundary-degen}
\sum_{\y \in \T_\alpha \cap \T_\beta} \sum_{\substack{\phi_1 \in \pi_2(\x,\y) \\ \phi_2 \in \pi_2(\y,\x) \\ \phi_1 * \phi_2 = \phi}} \#\hat\MM(\phi_1) \, \#\hat\MM(\phi_2) =
\begin{cases}
  -1 & \phi \in \pi_2^\alpha(\x) \\
  1 & \phi \in \pi_2^\beta(\x). \qedhere
\end{cases}
\end{equation}
\end{lemma}

We also mention the behavior of $\HFKtil$ under adding a marked point. Let $V = \Z \oplus \Z$, with generators in bigradings $(0,0)$ and $(-1,-1)$. We have:

\begin{lemma} \label{lemma: HFK-samecomponent}
Let $\LL = (L,\p)$ be a pointed link, and suppose $\p$ contains two points $p_0,p_1$ that are on the same component of $L$. Let $\p' = \p \minus \{p_0\}$, and set $\LL' =(L, \p')$. Then
\[
\HFKtil(\LL;\Z) \cong \HFKtil(\LL';\Z) \otimes V.
\]
In particular,
\[
\HFKtil(\LL;\Z) \cong \HFK(L;\Z) \otimes V^{\otimes \abs{\p}-l},
\]
and for any field $\F$,
\[
\rank \HFKtil{}^\delta(\LL;\F) = 2^{\abs{\p}-l} \rank \HFK{}^\delta(L;\F).
\]
\end{lemma}

\begin{proof}
This is just \cite[Theorem 2.5]{Sar-signs}, together with the observation that the bijection of weak equivalence classes of orientation systems discussed there respects the canonical systems.
\end{proof}

\subsection{The homology and basepoint actions} \label{subsec: HFK-action}

We now describe two additional algebraic structures on
$\HFKtil(\LL;\Z)$: the $H_1$ action defined by Ni
\cite{NiActions}, and the \emph{basepoint
  action}, a generalization of \cite[Section
3.3]{BaldwinLevineSpanning} and \cite[Section 4]{Sar-action}. Let
$\HH$ be a Heegaard diagram as in the previous section.

Note that we observe the sign conventions for chain maps coming from Section \ref{sec: khovanov}, using the Maslov grading in place of $\homgr$. Specifically, a chain map $f$ on $\CFKtil(\HH)$ whose Maslov grading shift is odd is required to satisfy $f \circ \partial + \partial \circ f = 0$. For two such maps $f,g$, let $[f,g]$ denote their anti-commutator $[f,g] = f \circ g+g \circ f$.

For any immersed $1$-chain $\zeta$ in $\Sigma$ that is in general position with respect to $\bm\alpha \cap \bm\beta$ and has endpoints in $\w \cup \z$, and any Whitney disk $\phi \in \pi_2(\x,\y)$, let $\zeta(\phi)$ denote the intersection number $\zeta \cdot \partial_\alpha \phi$. Adapting the work of Ozsv\'ath and Szab\'o \cite[Section 4.2.5]{OSz3Manifold} for closed $3$-manifolds, Ni defined a linear map $A^\zeta \co \CFKtil(\HH) \to \CFKtil(\HH)$ by
\begin{equation} \label{eq: HFK-Azeta}
A^\zeta(\x) = \sum_{\y \in \T_\alpha \cap \T_\beta} \sum_{ \substack{\phi \in \pi_2(\x, \y) \\ \mu(\phi)=1 \\ \mults_\w(\phi) = \mults_\z(\phi) = 0 }} \#\widehat{\MM}(\phi) \, \zeta(\phi) \, \x.
\end{equation}
Ni proved that $A^\zeta$ is an chain map whose square is zero. Moreover, up to chain homotopy, $A^\zeta$ depends only on the homology class of $\zeta$ in $H_1(Y(L), \partial Y(L))$, which by excision is isomorphic to $H_1(S^3, L)$. Thus, there is an action of $\Lambda^*(H_1(S^3,L))$ on $\HFKtil(\LL)$.

For each $p \in \p$, let $\Psi^p$ be the map that counts disks that go once over $z_p$ and zero times over the other basepoints:
\begin{align}
\label{eq: HFK-Psi}
\Psi^p(\x) = \sum_{\y \in \T_\alpha \cap \T_\beta} \sum_{ \substack{\phi \in \pi_2(\x, \y) \\ \mu(\phi)=1 \\ \mults_\z(\phi) = p \\ \mults_\w(\phi)=0}} \#\widehat{\MM}(\phi) \, \x.
\end{align}

Analogously, let $\Omega^p$ count disks that go once over $w_p$ and zero times over the other basepoints. Observe that the bigrading shifts of $A^\zeta$, $\Psi^p$, and $\Omega^p$ are $(-1,0)$, $(-1,-1)$ and $(1,1)$, respectively. We have:

\begin{proposition} \label{prop:HFK-action}
\begin{enumerate}
\item \label{item: HFK-action: Psi-chain}
The maps $\Psi^p$ and $\Omega^p$ are chain maps whose squares are chain null-homotopic.

\item \label{item: HFK-action: comm}
For any $p, p' \in \p$ and any $1$-chain $\zeta$ as above, the following relations hold up to chain homotopy:
\begin{align}
\label{eq: Psi-comm}
[\Psi^p, \Psi^{p'}] &\sim [\Omega^p, \Omega^{p'}] \sim 0 \\
\label{eq: Psi-Azeta}
[\Psi^p, A^\zeta] &\sim [\Omega^p, A^\zeta] \sim 0 \\
\label{eq: Psi-Omega}
[\Psi^p, \Omega^{p'}] &\sim (\delta_{\nextpt(p),p'} - \delta_{p,p'}) \id,
\end{align}
where $\delta$ denotes the Kronecker delta.
\end{enumerate}
\end{proposition}

\begin{proof}
Each of these statements follows from a standard degeneration argument; most details are left to the reader.
We shall focus on \eqref{eq: Psi-Omega}. Let $H$ be the map that counts holomorphic disk $\phi$ with $\mu(\phi)=1$, $\mults_\z(\phi) = p$, and $\mults_\w(\phi) = p'$. Consider the moduli spaces of index-$2$ disks with the same basepoint conditions. In addition to the terms that contribute
\[
H \circ \partial + \partial \circ H + \Psi^p \circ \Omega^{p'} + \Omega^{p'} \circ \Psi^p,
\]
there may be additional boundary degeneration terms if $z_p$ and $w_{p'}$ lie in the same component of $\Sigma \minus \bm\alpha$ or $\Sigma \minus \bm\beta$, which occurs precisely when $p' = p$ or $p' = \nextpt(p)$, respectively. The signs of these terms are determined by Lemma \ref{lemma: boundary-degen}. \end{proof}

For reasons that will become clear in Section \ref{sec: cube}, we will mostly ignore the $\Omega^p$ maps and focus only on the $\Psi^p$ maps. Proposition \ref{prop:HFK-action} thus gives $\HFK(\LL)$ the structure of a left $\Lambda_\p \otimes \Lambda^*H_1(S^3,L)$--module, where $y_p \in \Lambda_\p$ acts by $(\Psi^p)_*$. Furthermore, by studying the effects of isotopies, handleslides, and stabilizations, we see as in \cite[Proposition 3.6]{BaldwinLevineSpanning} that this module structure is a link invariant.

%
%
%
%
%
%

\begin{remark} \label{rmk: A-infty}
The $\Lambda_\p$-module structure of $\HFKtil(\LL)$ is the consequence of a more complicated algebraic structure on  $\CFKtil(\HH)$. We work here with coefficients in $\Z_2$ for convenience. For any formal linear combination $\a = \sum_{p \in \p} a_p \cdot p$ with nonnegative integer coefficients, define a map $\Psi^\a$ by
\begin{equation} \label{eq: Psi-higher}
\Psi^\a(\x) = \sum_{\y \in \T_\alpha \cap \T_\beta} \sum_{ \substack{\phi \in \pi_2(\x, \y) \\ \mu(\phi)=1 \\ \mults_\z(\phi) = \a \\ \mults_\w(\phi)= 0 }} \#\widehat{\MM}(\phi) \, \y.
\end{equation}
That is, $\Psi^\a$ counts disks with $n_{z_p}(\phi) = a_p$ which avoid all the $w$ basepoints. In particular, $\Psi^0$ is the differential, $\Psi^{2p}$ is the null-homotopy of $(\Psi^p)^2$, and $\Psi^{p+p'}$ is the null-homotopy of $[\Psi^p, \Psi^{p'}]$. These maps satisfy an $A_\infty$-type relation: for any $\a$, we have
\[
\sum_{\b+\c=\a} \Psi^\b \circ \Psi^\c = 0.
\]

Therefore, if we define a non-commutative dg-algebra $\mc{A}_{\p}$
generated freely by $\{ y_\a \mid \a \in (\Z_{\ge 0})^\p \minus \{0\}\} $, with differential given by
\[
\del y_{\a}=\sum_{\substack{\b+\c=\a \\ \b,\c\in (\Z_{\ge0})^\p \minus \{0\}}} y_{\b} y_{\c},
\]
then $\CFKtil(\HH;\Z_2)$ becomes a dg-module over $\mc{A}_{\p}$, where $y_\a$ acts by $\Psi^{\a}$.

By homological perturbation theory, $\CFKtil(\HH;\Z_2)$ can be endowed
with an $A_{\infty}$-module structure over the homology of
$(\mc{A}_{\p},\del)$, which is precisely the exterior algebra
$\Lambda_\p$. This module structure is supposed to be the Koszul dual
to the complex $\operatorname{gCFK}^-(\HH;\Z_2)$, defined over a
polynomial ring $\Z_2[U_1, \dots, U_m]$, in which the differential
counts disks that are allowed to go over the $z$ basepoints and
records these multiplicities with the $U_i$ powers.
\end{remark}

\subsection{Unreduced knot Floer homology}

Given a pointed link $\LL = (L,\p)$, let $\hat\LL = (\hat L, \hat \p)$ be the split union of $L$ with an unknotted component $L_0$ containing a single point $p_0$. We refer to $\HFKtil(\hat\LL)$ as the \emph{unreduced knot Floer homology} of $\LL$. Let $W = \Z \oplus \Z$, with generators in bigradings $(\frac12,0)$ and $(-\frac12,0)$.

Observe that $\hat\LL$ can always be represented with a Heegaard diagram in which $w_{p_0}$ and $z_{p_0}$ are in the same region. As a result, we immediately have $\Psi^{p_0} = \Omega^{p_0} = 0$. Thus, the action of $\Lambda_{\hat \p}$ on $\HFKtil(\hat \LL)$ is completely given by the action of $\Lambda_\p \subset \Lambda_{\hat \p}$.

\begin{lemma} \label{lemma: HFK-unred}
For any non-degenerate pointed link $\LL$, there is an isomorphism of $\Lambda_\p$--modules
\[
\HFKtil(\hat\LL) \cong \HFKtil(\LL) \otimes W.
\]
Therefore, for any field $\F$,
\begin{align*}
\rank \HFKtil{}^\delta(\hat\LL;\F) &= \rank \HFKtil{}^{\delta+1/2} (\LL;\F) + \rank \HFKtil{}^{\delta-1/2} (\LL;\F) \\
&= 2^{\abs{\p}-l} \left( \rank \HFK{}^{\delta+1/2} (L;\F) +
\rank \HFK{}^{\delta-1/2} (L;\F) \right).
\end{align*}
\end{lemma}

\begin{proof}
This follows by identifying $\hat\LL$ with the connected of sum of $\LL$ and a two-component unlink and applying the K\"unneth formula for connected sums. To be precise about signs, let $\HH = (\Sigma, \bm\alpha, \bm\beta, \w, \z)$ be a Heegaard diagram for $\LL$. Let $\HH_0 = (S^2, \alpha_0, \beta_0, \{w_0,w_1\}, \{z_0,z_1\})$ be a planar Heegaard diagram such that $\alpha_0$ and $\beta_0$ intersect in two points $a,b$ and divide $S^2$ into four bigons regions, one of which contains $w_0$ and $z_0$ and one of which contains $w_1$ and $z_1$, so that there are two positive classes $\phi_1, \phi_2 \in \pi_2(a,b)$ that avoid the basepoints. (This is the $m=1$ case of Figure \ref{fig: unlink}.) Let $\hat\HH$ be obtained by taking the connected sum of $\HH$ and $\HH_0$, where the connected sum joins the region of $\HH$ containing some basepoint $w_q$ (where $q \in \p$) and the region of $\HH_0$ containing $w_1$ and $z_1$. A standard argument (see, e.g., \cite[Section 6]{OSzProperties}) shows that there is an isomorphism of chain complexes
\[
\CFKtil(\hat\HH;\Z) \cong \CFKtil(\HH;\Z) \otimes \CFKtil(\HH_0;\Z).
\]
Moreover, the maps $\Psi^p$ (for $p \in \p$) respect this decomposition, since they count disks that avoid the connected sum region.

Clearly $\CFKtil(\HH_0;\Z) = \Z a \oplus \Z b$; we must simply verify that the differential on $\CFKtil(\HH_0;\Z)$ vanishes provided that we use the canonical system of orientations. Clearly, $\partial(a) = (\#\hat\MM(\phi_1) + \#\hat\MM(\phi_2)) \, b$ and $\partial(b)=0$. To see that $\partial(a) = 0$, let $\phi_3 \in \pi_2(b,a)$ be the bigon with $n_{w_0} = n_{z_0} = 1$, which has a unique holomorphic representative. By Lemma \ref{lemma: boundary-degen}, we have
\[
\#\hat\MM(\phi_1) \, \#\hat\MM(\phi_3) = 1 \quad \text{and} \quad \#\hat\MM(\phi_2) \, \#\hat\MM(\phi_3) = -1,
\]
so $\#\hat\MM(\phi_1) = -\#\hat\MM(\phi_2)$, as required.
\end{proof}

If the components of $L$ are denote $L_1, \dots, L_l$, let $\zeta_i \in H_1(S^3, \hat L)$ denote the relative homology class of an arc from $L_0$ (the added component of $\hat \LL$) to $L_i$ for $i=1, \dots, l$. Then the exterior algebra $\Lambda^* H_1(S^3,\hat L)$, which acts on $\HFKtil(\hat \LL)$, can be identified with the algebra $\Gamma_L$ from Section \ref{sec: Kh-unlink}. (In that section, $L$ was assumed to be an unlink, but the definition of $\Gamma_L$ makes sense for any $L$.)

\subsection{Unlinks}

We now give a complete description of the module structure of unreduced knot Floer homology for unlinks, which will be needed in discussing the cube of resolutions.

To begin, if $\LL = (L,\p)$ is an $l$-component unlink with $m$ marked points, induction using Lemmas \ref{lemma: HFK-samecomponent} and \ref{lemma: HFK-unred} shows that as a bigraded abelian group,
\[
\HFKtil(\hat \LL; \Z) \cong W^{\otimes l} \otimes V^{\otimes m-l}
\]
The summand in maximal Maslov and Alexander gradings is a $\Z$ in bigrading $(\frac{l}{2},0)$. Let $x_{\max}$ denote a generator of this summand group. Ignoring gradings, we see that $\HFKtil(\hat\LL;\Z) \cong \Z^{2^m}$, irrespective of the number of components of $L$.

We now describe the $\Lambda_\p \otimes \Gamma_L$--module structure on $\HFKtil(\hat\LL; \Z)$.

\begin{proposition} \label{prop: HFK-unlink}
Let $\LL = (L, \p)$ be a non-degenerate, oriented, $k$-component pointed unlink. Then $\HFKtil(\hat\LL;\Z)$ is isomorphic as a $\Lambda_\p \otimes \Gamma_L$--module to $\robar \Lambda_{L,\p} \otimes \Gamma_L$, generated by an element $x_{\max}$ in bigrading $(\frac{k}{2},0)$. In particular, $\HFKtil(\hat\LL;\Z)$ is canonically (up to an overall sign) isomorphic to $\Kh(L,\p)$.
\end{proposition}

%

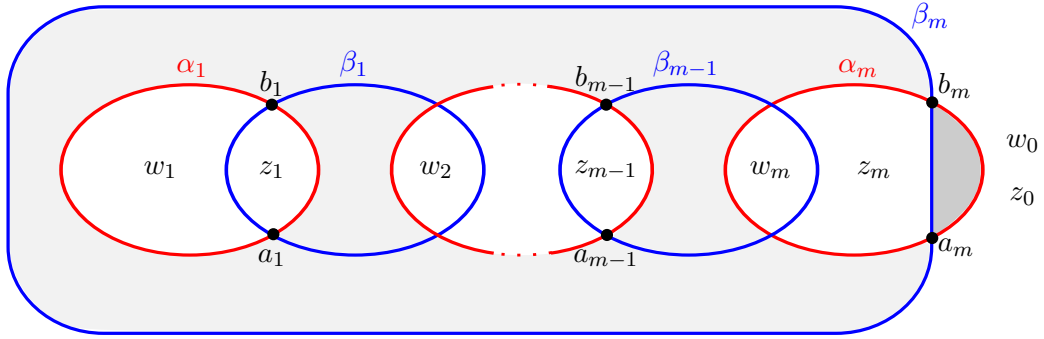
\begin{figure}
 \labellist
 \small
 \pinlabel {{\color{red} $\alpha_1$}} [b] at 78 103
 \pinlabel {{\color{blue} $\beta_1$}} [b] at 139 103
 \pinlabel {{\color{blue} $\beta_{m-1}$}} [b] at 264 103
 \pinlabel {{\color{red} $\alpha_{m}$}} [b] at 329 103
 \pinlabel {{\color{blue} $\beta_{m}$}} [bl] at 347 122
 \pinlabel $w_1$ at 66 71
 \pinlabel $z_1$ at 108 71
 \pinlabel $w_2$ at 170 71
 \pinlabel $z_{m-1}$ at 234 71
 \pinlabel $w_{m}$ at 296 71
 \pinlabel $z_{m}$ at 335 71
 \pinlabel $w_0$ at 391 81
 \pinlabel $z_0$ at 391 61
 \pinlabel $b_1$ [b] at 108 96
 \pinlabel $a_1$ [t] at 108 44
 \pinlabel $b_{m-1}$ [b] at 234 97
 \pinlabel $a_{m-1}$ [t] at 234 44
 \pinlabel $b_{m}$ [bl] at 357 96
 \pinlabel $a_{m}$ [tl] at 357 47
\endlabellist
\begin{center}
\includegraphics{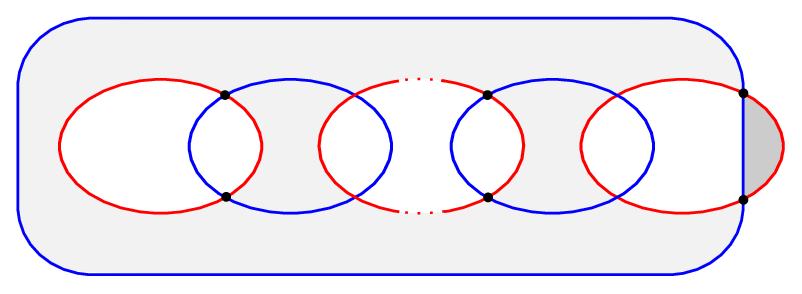}
\caption{Heegaard diagram for an $m$-pointed unknot. A Heegaard diagram for a $k$-component unlink consists of $k$ diagrams of this form, sharing the region containing $w_0$ and $z_0$.} \label{fig: unlink}
\end{center}
\end{figure}

\begin{proof}
First, consider the case where $L$ has a single component, and assume that the points of $\p$ are labeled $1, \dots, m$ according to the orientation of $L$. We may represent $\hat\LL$ with a Heegaard diagram $\HH$ as shown in Figure \ref{fig: unlink}. The basepoints $w_0, z_0$ corresponding to the extra component are placed in the unbounded region. Let $\zeta$ be an arc from $w_0$ to $z_m$ that passes through the dark shaded bigon in Figure \ref{fig: unlink}; this arc represents a generator of $H_1(S^3, \hat L)$.

A generator $\x$ of $\uCFK(\HH)$ consists of a choice of either $a_j$ or $b_j$ for $j = 1,\dots,m$. In particular, the rank of the complex $\uCFK(\HH)$ is $2^m$, which agrees with the rank of its homology, and the differential vanishes. The unique generator in the maximal Alexander and Maslov gradings, $\x_{\max}$, consists of all the points $a_j$.

For any generator $\x$ containing the point $a_m$, there are a pair of Whitney disks connecting $\x$ to $\x' = \x \minus \{a_m\} \cup \{b_m\}$ and avoiding all basepoints, each with a unique holomorphic representative; the domains of these disks are the heavily shaded and lightly shaded regions in Figure \ref{fig: unlink}. Since the differential vanishes, the $\Psi^j$ and $\Omega^j$ maps satisfy the commutation relations from Proposition \ref{prop:HFK-action} on the nose, and not merely up to chain homotopy. In particular, observe that $\Psi^1 + \cdots + \Psi^m$ anticommutes with $\Omega^j$ for each $j=1, \dots, m$. We claim, in fact, that $\Psi^1 + \cdots + \Psi^m=0$, and that this is the only linear relation satisfied by $\Psi^1, \dots, \Psi^m$.

It is not hard to verify directly that these maps are given as follows. Observe that the only domains that count for the maps $\Psi^1, \dots, \Psi^{m-1}$ and $A^{\zeta}$ are small bigons, which have unique holomorphic representatives. Specifically, for each $j=1, \dots, m$, define a linear map $\tau_j$ on $\uCFK(\bm\alpha, \bm\beta)$ by
\begin{align}
\label{eq: unlink: tauj}
\tau_j(\x) =
\begin{cases}
\x \minus \{a_j\} \cup \{b_j\} & a_j \in \x \\ 0  & a_j \not\in \x
\end{cases} 
\end{align}
For each generator $\x$, we then have
\begin{align}
\label{eq: unlink: Psij}
\Psi^j(\x) &= \pm \tau_j(\x) \qquad \text{for }j=1, \dots, m-1 \\
\label{eq: unlink: Azeta}
A^{\zeta}(\x) &= \pm \tau_m(\x) 
\end{align}
where each sign may depend on the choice of $\x$. Moreover, $\Psi^m(\x)$ is a linear combination of $\tau_1(\x), \dots, \tau_{m-1}(\x)$, and there is a positive domain corresponding to each $j$ for which $\tau_j(\x) \ne 0$.

To resolve the sign ambiguities in $A^\zeta$ and $\Psi^1, \dots, \Psi^{m-1}$, we identify $\uCFK(\bm\alpha, \bm\beta)$ with an exterior algebra $\Lambda^*(B_1, \dots, B_m)$, taking the generator $\x_{\max}$ to $1$, such that $\Psi^j$ is given by left multiplication by $B_j$ for $j=1, \dots, m-1$, and $A^\zeta$ is given by left multiplication by $B_m$. Using this notation, we claim that $\Psi^m$ is given by left multiplication by $-(B_1 + \cdots + B_{m-1})$. Since $\Psi^m$ anticommmutes with $\Psi^1, \dots, \Psi^{m-1}$ and $A^\zeta$, it suffices to show that $\Psi^m(1) = -(B_1 + \cdots + B_{m-1})$. Write
\[
\Psi^m(1) = c_1 B_1 + \cdots + c_{m-1} B_{m-1},
\]
for some coefficients $c_1, \dots, c_{m-1}$, so that
\begin{equation} \label{eq: unlink: sumPsi(1)}
(\Psi^1 + \cdots + \Psi^m)(1) = (1+c_1) B_1 + \cdots + (1+c_{m-1}) B_{m-1}.
\end{equation}

We now use the $\Omega^i$ maps. First, observe that $\Omega^i(1) = 0$, since $\x_{\max}$ is in maximal Alexander grading. Therefore, for $i=1, \dots, m$ and $j=1, \dots, m-1$, \eqref{eq: Psi-Omega} shows that
\begin{equation} \label{eq: unlink: Omegai(Bj)}
\Omega^i(B_j) = \Omega^i(\Psi^j(1)) =
\begin{cases}
-1 & i = j \\
1 & i = j+1 \\
0 & \text{otherwise}.
\end{cases}
\end{equation}
Moreover, $\Omega^i$ anticommutes with $\Psi^1 + \cdots + \Psi^m$, so
\[
\Omega^i \circ (\Psi^1 + \cdots + \Psi^m)(1)=0.
\]
It immediately follows that $c_1 = \cdots = c_{m-1} = -1$, as required.

Now, for the general case where $L$ is an $k$-component unlink, we may represent $L$ with a Heegaard diagram $H$ that is the connected sum of $k$ diagrams $H_1, \dots, H_k$ that are each as in Figure \ref{fig: unlink}, taking the connected sums in the regions containing $\{w_0,z_0\}$. The K\"unneth formula then indicates that
\[
\uCFK(\HH) \cong \uCFK(\HH_1) \otimes \cdots \otimes \uCFK(\HH_k).
\]
Moreover, the $\Psi$, $\Omega$, and $A^\zeta$ maps all respect this decomposition, since they only count disks that avoid the shared region. The result then follows easily.
\end{proof}

\subsection{Elementary cobordisms} \label{subsec: HFK-cobordism}

We now consider elementary cobordisms between planar unlinks. Just as in Section \ref{subsec: Kh-cobordism}, suppose that $\LL = (L,\p)$ is an oriented, non-degenerate planar unlink with components $L_1, \cdots, L_k$, and let $\LL' = (L',\p)$, where $L'$ is obtained by splitting $L_1$ into components $L'_0 \cup L'_1$. Assume further that $(L',\p)$ is also non-degenerate (meaning that $L_1$ contains at least two points of $\p$). As above, we write  $L'_i = L_i$ for $i=2, \dots, k$.


For concreteness, let us assume that the points of $\p$ on $L_i$ are labeled $p^i_1, \dots, p^i_{m_i}$ according to the orientation of $L$. Assume further that the split occurs on the segments $[p^1_l, p^1_{l+1}]$ and $[p^1_{m_1}, p^1_1]$, so that $\p \cap L'_0 = \{p^1_1, \dots, p^1_l\}$ and $\p \cap L'_{1} = \{p^1_{l+1}, \dots, p^1_m\}$.

The decorated saddle cobordism from $L$ to $L'$ (a surface in $S^3 \times [0,1]$) naturally gives rise to a sutured cobordism $W$ from $Y(\hat\LL)$ to $Y(\hat\LL')$. To represent $W$ with a Heegaard diagram, let $(S^2,\bm\alpha,\bm\beta)$ be a Heegaard diagram for $\hat \LL$ just as in the proof of Proposition \ref{prop: HFK-unlink}; the portion of the diagram corresponding to $L_1$ is as shown in Figure \ref{fig: unlink}. Let $\gamma$ be the multicurve defined as follows: the circles $\beta_l$ and $\beta_m$ are replaced with a pair of $\gamma$ circles as shown in Figures \ref{fig: split} and \ref{fig: merge}, and every other $\gamma$ circle is simply a small isotopic translate of the corresponding $\beta$ circle.

We shall denote the $w$ and $z$ basepoints corresponding to $p^i_j$ by $w^i_j$ and $z^i_j$, and the corresponding basepoint maps by $\Omega^i_j$ and $\Psi^i_j$.

\begin{figure}
\labellist
 \small
 \pinlabel $w^1_1$ at 58 60
 \pinlabel $z^1_1$ at 94 60
 \pinlabel $w^1_2$ at 148 60
 \pinlabel $z^1_l$ at 204 60
 \pinlabel $w^1_{l+1}$ at 252 60
 \pinlabel $z^1_{m_1-1}$ at 309 60
 \pinlabel $w^1_{m_1}$ at 364 60
 \pinlabel $z^1_{m_1}$ at 395 60
 \pinlabel $\Theta^0_{\beta\gamma}$ [tl] at 213 37
 \pinlabel $\Theta^1_{\beta\gamma}$ [bl] at 213 84
 \large
 \pinlabel {{\color{blue} $\bullet$}} at 86 38
 \pinlabel {{\color{blue} $\bullet$}} at 205 41
 \pinlabel {{\color{blue} $\bullet$}} at 300 38
 \pinlabel {{\color{blue} $\bullet$}} at 407 38
 \pinlabel $\bullet$ at 73.5 61
 \pinlabel $\bullet$ at 214.5 37
 \pinlabel $\bullet$ at 289 61
 \pinlabel $\bullet$ at 404 102
 \pinlabel $\circ$ at 214.5 84
 \pinlabel {{\color{darkgreen} $\bullet$}} at 99 43
 \pinlabel {{\color{darkgreen} $\bullet$}} at 214.5 74
 \pinlabel {{\color{darkgreen} $\bullet$}} at 314 43
 \pinlabel {{\color{darkgreen} $\bullet$}} at 417.5 43
\endlabellist
\includegraphics[width=5in]{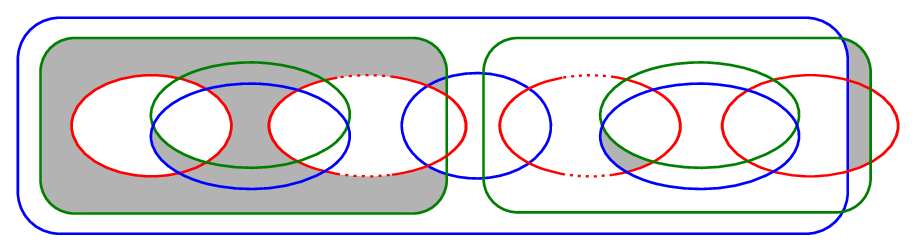}
\caption{Part of the Heegaard diagram $(S^2, \bm\alpha, \bm\beta, \bm\gamma, \w, \z)$ for a split. The points marked {\color{blue} $\bullet$} and {\color{darkgreen} $\bullet$} indicate the generators $\x_{\max} \in \T_\alpha \cap \T_\beta$ and $\y' = \pm A^{\zeta_0}(\y_{\max}) \in \T_\alpha \cap \T_\gamma$, respectively. The points marked $\bullet$ indicate $\Theta^0_{\beta\gamma} \in \T_\beta \cap \T_\gamma$, while $\Theta^1_{\beta\gamma}$ consists of all but one of the same points, together with the point marked $\circ$. The shaded area indicates the domain of $\psi_1 \in \pi_2(\x_{\max}, \Theta^{\beta\gamma}_0, \y')$.} \label{fig: split}
\end{figure}

\begin{figure}
\labellist
 \small
 \pinlabel $w^1_1$ at 58 60
 \pinlabel $z^1_1$ at 94 60
 \pinlabel $w^1_2$ at 148 60
 \pinlabel $z^1_l$ at 204 60
 \pinlabel $w^1_{l+1}$ at 252 60
 \pinlabel $z^1_{m_1-1}$ at 309 60
 \pinlabel $w^1_{m_1}$ at 364 60
 \pinlabel $z^1_{m_1}$ at 395 60
 \pinlabel $\Theta^1_{\gamma\beta}$ [tl] at 213 37
 \pinlabel $\Theta^0_{\gamma\beta}$ [bl] at 213 84
 \large
 \pinlabel {{\color{blue} $\bullet$}} at 86 38
 \pinlabel {{\color{blue} $\bullet$}} at 205 80
 \pinlabel {{\color{blue} $\bullet$}} at 300 38
 \pinlabel {{\color{blue} $\bullet$}} at 407 38
 \pinlabel $\bullet$ at 167 61
 \pinlabel $\bullet$ at 214.5 84
 \pinlabel $\bullet$ at 382 61
 \pinlabel $\bullet$ at 404 19
 \pinlabel $\circ$ at 214.5 37
 \pinlabel {{\color{darkgreen} $\bullet$}} at 99 43
 \pinlabel {{\color{darkgreen} $\bullet$}} at 214.5 46
 \pinlabel {{\color{darkgreen} $\bullet$}} at 314 43
 \pinlabel {{\color{darkgreen} $\bullet$}} at 417.5 43
\endlabellist
\includegraphics[width=5in]{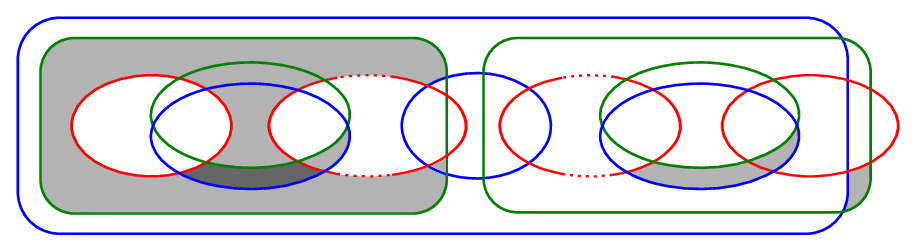}
\caption{Part of the Heegaard diagram $(S^2, \bm\alpha, \bm\gamma, \bm\beta, \w, \z)$ for a merge. The points marked {\color{blue} $\bullet$} and {\color{darkgreen} $\bullet$} indicate the generators $\x' = \pm \Psi^1_l(\x_{\max}) \in \T_\alpha \cap \T_\beta$ and $\y_{\max}$, respectively. The points marked $\bullet$ indicate $\Theta^0_{\gamma\beta} \in \T_\gamma \cap \T_\beta$, while $\Theta^1_{\gamma\beta}$ consists of all but one of the same points, together with the point marked $\circ$.
The shaded area indicates the domain of $\psi_0' \in \pi_2(\y_{\max}, \Theta^0_{\gamma\beta}, \x')$; the darkly shaded region has multiplicity $2$.} \label{fig: merge}
\end{figure}
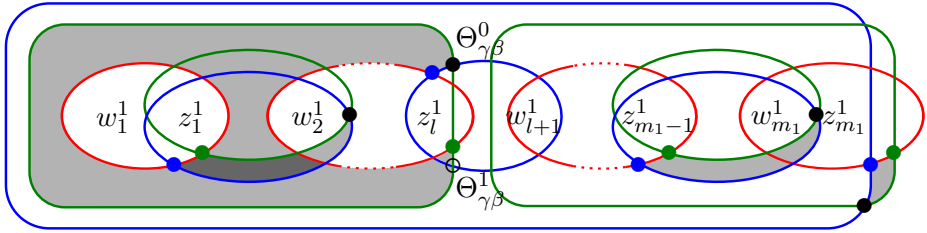

For $i=1,\dots, k$, let $\zeta_i$ be an arc joining $w_0$ and $z^i_{m_i}$, as in the proof of Proposition \ref{prop: HFK-unlink}, and let $\zeta_0$ be an arc joining $w_0$ and $z^1_l$. The homology classes of $\zeta_1, \dots, \zeta_{k}$ form a basis for $H_1(S^3, \hat L)$, while $[\zeta_0]=[\zeta_1]$; likewise, the homology classes of $\zeta_0, \dots, \zeta_{k}$ form a basis for $H_1(S^3, \hat L')$.

To work with coefficients in $\Z$, we must specify a system of orientations --- namely a choice of trivialization of the determinant line bundle for every disk or triangle. We claim that we can choose this system to agree with the canonical choices for $\CFKtil(\bm\alpha, \bm\beta)$, $\CFKtil(\bm\beta, \bm\gamma)$, and $\CFK(\bm\alpha,\bm\gamma)$ simultaneously. The key observation is that every triply periodic domain is a linear combination of singly periodic domains. Thus, just as above, we choose orientations for the boundary degenerations such that the moduli space of each component of $S^2 \minus \bm\alpha$, $S^2 \minus \bm\beta$, or $S^2 \minus \bm\gamma$ has count $+1$; this determines the orientation for every periodic domain. The rest of the argument then follows as in \cite[Lemma 8.7]{OSz3Manifold}.

Observe that $\uCFK(\bm\beta, \bm\gamma)$ is supported in two adjacent Alexander gradings and has vanishing differential. There are two distinguished generators $\Theta^0_{\beta\gamma}, \Theta^1_{\beta\gamma} \in \uCFK(\bm\beta, \bm\gamma)$, each of which has maximal Maslov grading within its Alexander grading, such that $A(\Theta^1_{\beta\gamma}) = A(\Theta^0_{\beta\gamma})+1$. Define distinguished cycles $\Theta^0_{\gamma\beta}, \Theta^1_{\gamma\beta} \in \uCFK(\bm\gamma, \bm\beta)$ analogously.

Now, for $a \in \{0,1\}$, let
\begin{align*}
f^a \co \uHFK(L, \p) &\to \uHFK(L', \p) \\
g^a \co \uHFK(L', \p) &\to \uHFK(L, \p)
\end{align*}
be the maps given by
\begin{align*}
f^a(\x) &= F_{\alpha\beta\gamma}(\x \otimes \Theta^a_{\beta\gamma}) \\
g^a(\x) &= F_{\alpha\gamma\beta}(\x \otimes \Theta^a_{\gamma\beta}),
\end{align*}
where we have identified the complexes $\uCFK(\bm\alpha, \bm\beta)$ and $\uCFK(\bm\alpha, \bm\gamma)$ with their homologies since the differentials vanish. Our goal is to compute the maps $f^a$ and $g^a$ in terms of the module descriptions given by Proposition \ref{prop: HFK-unlink}. We shall primarily be interested in $f^0$ and $g^0$, but we shall also give a description of $f^1$ and $g^1$.

\begin{proposition} \label{prop: HFK-edgemaps0}
Let $\LL=(L,\p)$ and $\LL'=(L',\p)$ be pointed planar unlinks as described above, and assume that both are non-degenerate. Choose module identifications $\HFKtil(\hat\LL) \cong \robar \Lambda_{L,\p} \otimes \Gamma_L $ and $\HFKtil(\hat\LL{}',\p) \cong \robar \Lambda_{L',\p} \otimes \Gamma_{L'}$ as in Proposition \ref{prop: Kh-unlink}. Then the maps $f^0$ and $g^0$ agree up to sign with the maps $\tilde \Delta_*$ and $\tilde \mu_*$ described in Proposition \ref{prop: Kh-edgemaps}. In particular,
\begin{align}
\label{eq: f0(11)}
f^0(1 \otimes 1) &= \pm 1 \otimes (\zeta_0 - \zeta_1)  \\
\label{eq: g0(11)}
g^0(1 \otimes 1) &= \alpha'_0 \otimes 1 = \mp \alpha'_1 \otimes 1
\end{align}
(using notation from Section \ref{subsec: Kh-cobordism}).
\end{proposition}

\begin{proof}
Because the $\Psi$ maps drop Alexander grading by $1$, we see that $\Psi^i_j(\Theta^0_{\beta\gamma})=0$ and $\Psi^i_j(\Theta^0_{\gamma\beta})=0$ for any $i=1, \dots, k$ and $j=1, \dots, m_i$. It follows that $f^0$ and $g^0$ commute with $\Psi^i_j$ for any $i,j$. (Compare \cite[Equation 3.14]{BaldwinLevineSpanning}.) Additionally, standard degeneration arguments show that $f^0$ and $g^0$ commute with $A^{\zeta_i}$ for $i=0, \dots, k$; that is, they intertwine the module structures of $\CFKtil(\bm\alpha, \bm\beta)$ and $\CFKtil(\bm\alpha, \bm\gamma)$.

Let $\x_{\max} \in \T_\alpha \cap \T_\beta$ and $\y_{\max} \in \T_\alpha \cap \T_\gamma$ denote the unique generators in maximal Alexander and Maslov gradings; these are indicated by the blue and green dots in Figure \ref{fig: merge}, respectively. Since these generate $\HFKtil(\hat\LL)$ and $\HFKtil(\hat\LL')$ as $\Lambda_\p \otimes \Gamma_{L'}$--modules, it suffices to compute $f^0(\x_{\max})$ and $g^0(\y_{\max})$. Specifically, we claim that
\begin{align*}
f^0(\x_{\max}) &= \pm (A^{\zeta_0} - A^{\zeta_1})(\y_{\max}) \\
g^0(\y_{\max}) &= \pm (\Psi^1_1 + \cdots + \Psi^1_l)(\x_{\max}).
\end{align*}
Under our identifications, these translate into \eqref{eq: f0(11)} and \eqref{eq: g0(11)}, respectively.

We first consider the split map $f^0$. Because $A^{\zeta_0} = A^{\zeta_1}$ on $\uCFK(\bm\alpha, \bm\beta)$, we see that $f^0(\x)$ must be in the kernel of $A^{\zeta_0} - A^{\zeta_1}$ in $\uCFK(\bm\alpha, \bm\beta)$. Let $\y' \in \T_\alpha \cap \T_\gamma$ be the generator such that $A^{\zeta_0}(\y_{\max}) = \pm \y'$, and observe that there is a unique positive class $\psi_0 \in \pi_2(\x_{\max}, \Theta^0_{\beta\gamma}, \y')$, avoiding all the basepoints, with $\mu(\psi_0)=0$. (The domain of $\psi_0$ is shown in Figure \ref{fig: merge}.) Therefore, $f^0(\x_{\max})$ is a linear combination of the generators that are in the same bigrading as $\y'$ --- namely $A^{\zeta_0}(\y_{\max}), \dots, A^{\zeta_k}(\y_{\max})$. It follows that
\[
f^0(\x_{\max}) = \pm \#\MM(\psi_0) (A^{\zeta_0} - A^{\zeta_1})(\y_{\max}),
\]
since this is the only such linear combination that is in the kernel of $A^{\zeta_0} - A^{\zeta_1}$.

Rather than analyzing the moduli space of $\psi_0$ directly, consider the following homotopy classes of disks:
\begin{itemize}
\item
Let $\phi_0 \in \pi_2(\Theta^1_{\beta\gamma}, \Theta^0_{\beta\gamma})$ be the class whose domain is the bigon containing $z^1_l$; it has a unique holomorphic representative.

\item
Let $\psi_1 \in \pi_2(\x_{\max}, \Theta^1_{\beta\gamma}, \y')$ be the class whose domain is the triangle containing $z^1_l$ together with small triangles near the other points; it has a unique holomorphic representative.

\item
Let $\phi_1, \phi_2 \in \pi_2(\y_{\max},\y')$ be the two positive classes that avoid all the basepoints, analogous to the two shaded regions in Figure \ref{fig: unlink}, where $\phi_1$ is the small bigon and $\phi_2$ is the other domain. Since the contributions of these domains cancel in $\partial(\y')$, and $\phi_1$ has a unique holomorphic representative, we see that $\#\MM(\phi_2)=\pm 1$.
\end{itemize}
Observe that $\phi_0 * \psi_0 = \psi_1 * \phi_2$, and that these are the only two decompositions of the joined domain. Therefore,
\[
\#\widehat{\MM}(\phi_0) \cdot \#\MM(\psi_0) + \#\MM(\psi_1) \cdot \#\widehat{\MM}(\phi_2) = 0,
\]
so we deduce that $\#\MM(\psi_0)=\pm 1$, as required.

The case of the merge is similar. First, since $\Psi^1_1 + \cdots + \Psi^1_l = 0$ on $\uCFK(\bm\alpha, \bm\gamma)$, $g^0(\y_{\max})$ must be in the kernel of $\Psi^1_1 + \cdots + \Psi^1_l$. Let $\x'$ be the generator with $\Psi^1_l(\x_{\max}) = \pm \x'$. There is a unique positive domain $\psi_0' \in \pi_2(\y_{\max}, \Theta^0_{\gamma\beta}, \x')$ that avoids the basepoints and has $\mu(\psi_0')=0$; its domain is shown in Figure \ref{fig: merge}. Therefore, $g^0(\y_{\max})$ is a linear combination of the generators that are in the same bigrading as $\x'$ --- namely $\Psi^i_j(\x)$ for all $i=1, \dots, k$ and $j=1, \dots, m_i-1$. It follows that
\[
g^0(\y_{\max}) = \pm \#\MM(\psi_0') (\Psi^1_1 + \cdots + \Psi^1_l)(\x_{\max}),
\]
since this is the only such linear combination that is in the kernel of $\Psi^1_1 + \cdots + \Psi^1_l$. A degeneration argument similar to the one used above proves that $\#\MM(\psi_0') = \pm 1$.
\end{proof}

Propositions \ref{prop: HFK-unlink} and \ref{prop: HFK-edgemaps0} complete the proof of statement \ref{item: Kh-properties: unlink} of Theorem \ref{thm: Kh-properties}.

We now turn to $f^1$ and $g^1$. Since the $\Omega^i_j$ maps raise Alexander grading by $1$, we have $\Omega^i_j(\Theta^1_{\beta\gamma})=0$ and $\Omega^i_j(\Theta^1_{\gamma\beta})=0$ for any $i,j$. Therefore, $f^1$ and $g^1$ can be described very neatly in terms of the $\omega$ maps on $\HFKtil(\LL)$ and $\HFKtil(\LL')$, in a manner almost identical to Proposition \ref{prop: HFK-edgemaps0}. However, it can be more convenient to describe these maps using the same module structure as was used for $f^0$ and $g^0$ --- with the caveat that they are not $\Lambda_\p$--linear. For notational convenience, we only work with $\Z_2$ coefficients here.

\begin{proposition} \label{prop: HFK-edgemaps1}
With coefficients in $\Z_2$, the maps $f^1$ and $g^1$ are determined by the following properties:
\begin{enumerate}
\item \label{item: HFK-edgemaps1: H1}
For any $i=0, \dots, k$, $f^1$ and $g^1$ commute with $A^{\zeta_i}$.

\item \label{item: HFK-edgemaps1: psi}
For any $i=1, \dots, k$ and $j=1, \dots, m_i$,
\begin{equation} \label{eq: HFK-edgemaps1: psi}
[f^1, \Psi^i_j] =
\begin{cases}
f^0  & (i,j) = (1,l) \text{ or }  (1,m_1)\\
0 & \text{otherwise}.
\end{cases}
\end{equation}
An analogous equation holds with $g$ in place of $f$.

\item \label{item: HFK-edgemaps1: xmax}
$f^1(\x_{\max}) = 0$ and $g^1(\y_{\max}) = \x_{\max}$.
\end{enumerate}
\end{proposition}

\begin{proof}
Statement \ref{item: HFK-edgemaps1: H1} follows from a standard degeneration argument, just like in Proposition \ref{prop: HFK-edgemaps0}. For statement \ref{item: HFK-edgemaps1: psi}, a simple count of holomorphic disks shows that
\[
\Psi^i_j(\Theta^1_{\beta\gamma}) =
\begin{cases}
\Theta^0_{\beta\gamma} & (i,j) = (1,l) \text{ or }  (1,m_1)\\
0 & \text{otherwise},
\end{cases}
\]
and \eqref{eq: HFK-edgemaps1: psi} then follows from a degeneration argument as in \cite[Equation 7.5]{BaldwinLevineSpanning}.

For statement \ref{item: HFK-edgemaps1: xmax}, the Alexander grading of $f^1(\x_{\max})$ must be one greater than that of $f^0(\x_{\max})$, which has maximal Alexander grading; hence, $f^1(\x_{\max})=0$. Also, there is a unique positive class in $\pi_2(\y_{\max}, \Theta^1_{\gamma\beta}, \x_{\max})$ that avoids the basepoints; it has a unique holomorphic representative.
\end{proof}

\section{The cube of resolutions for knot Floer homology} \label{sec: cube}

We now consider the cube of resolutions for knot Floer homology, constructed in \cite{BaldwinLevineSpanning}. We find it convenient to rederive this complex in terms of a planar Heegaard diagram rather than the large-genus diagram used there; the former is better adapted to unreduced knot Floer homology. The results of \cite{BaldwinLevineSpanning} all go through, \emph{mutatis mutandis}; we merely note a few of the relevant differences.\footnote{The main notational changes from \cite{BaldwinLevineSpanning} are that we use $\bm\beta$, $\w$, $\z$, and $v$ in place of $\bm\eta$, $\mathbb{O}$, $\mathbb{X}$, and $I$, respectively, and sometimes swap the roles of superscripts and subscripts.} \emph{In this section, all Floer homology groups are taken with coefficients in $\Z_2$.}

\subsection{Construction of the cube} \label{subsec: cube-construction}

Let $L \subset S^3$ be an oriented link diagram with crossings $c_1, \dots, c_n$. Let $\p$ be a set of $2n$ points, one on each edge of the projection, and let $\LL = (L,\p)$. (It is easy to modify the construction if $\p$ contains more than one point on a given edge.) Let $n_+(L)$ and $n_-(L)$ denote the numbers of positive and negative crossings, respectively. For any $v \in \{0,1,\infty\}^n$, let $\LL_v = (L_v,\p)$ denote the pointed link obtained from $L$ by taking the $v_i$ resolution of $c_i$, where $v_i = \infty$ means leaving the crossing unresolved, and let $\hat\LL_v$ denote the split union of $\LL_v$ with a once-pointed unknot with one marked point, placed in the unbounded region of the plane.

\def\perturbshift{0.3}
\begin{figure}
\centering
\begin{tikzpicture}[scale=0.5]

\tikzstyle{reverseclip}=[insert path={(-5,-5) -- (5,-5) -- (5,5) -- (-5,5) -- (-5,-5)}]

\clip (-5-5*\perturbshift,-5-5*\perturbshift) rectangle (5+5*\perturbshift,5+5*\perturbshift);

\begin{scope}
  \clip[rounded corners] (-4,-3) rectangle (-2,3) [reverseclip];
  \clip[rounded corners] (-3,2) rectangle (3,4);
  \clip[rounded corners,
  rotate=-45,scale=1.2] (-1,-3) rectangle (1,3) [reverseclip];
\end{scope}

\begin{scope}
  \clip[rounded corners] (-4,-3) rectangle (-2,3);
  \clip[rounded corners] (-3,2) rectangle (3,4) [reverseclip];
  \clip[rounded corners,
  rotate=-45,scale=1.2] (-1,-3) rectangle (1,3) [reverseclip];
\end{scope}

\begin{scope}
  \clip[rounded corners] (-4,-3) rectangle (-2,3) [reverseclip];
  \clip[rounded corners] (-3,2) rectangle (3,4) [reverseclip];
  \clip[rounded corners,
  rotate=-45,scale=1.2] (-1,-3) rectangle (1,3);
\end{scope}

\begin{scope}
  \clip[rounded corners] (-4,-3) rectangle (-2,3) [reverseclip];
  \clip[rounded corners] (-3,2) rectangle (3,4) [reverseclip];
  \clip[rounded corners,
  rotate=-45,scale=1.2] (-1,-3) rectangle (1,3) [reverseclip];
  \clip (-2.5,-2.5) -- (2.5,2.5) -- (-2.5,2.5) -- (-2.5,-2.5);
\end{scope}

\begin{scope}[rotate=-45,scale=1.2]
\draw[thick,red,rounded corners] (-1.5,2) rectangle (1.5,10);
\draw[thick,red,rounded corners] (-1.5,-10) rectangle (1.5,-2);
\draw[thick,red,rounded corners] (-10,-1.5) rectangle (-2,1.5);
\draw[thick,red,rounded corners] (2,-1.5) rectangle (10,1.5);
\node[red,anchor=south west, inner sep=0] at (2,1.5) {$\alpha$};

\draw[thick,blue,rounded corners] (-1,-3) rectangle (1,3);
\node[blue,anchor=south east,inner sep=0] at (1,0) {$\bar\beta^1_c$};
\end{scope}

\draw[thick,blue,rounded corners] (-5-\perturbshift,-5-\perturbshift) rectangle (5-\perturbshift,5-\perturbshift);
\draw[thick,green!50!black,rounded corners] (-5,-5) rectangle (5,5);
\draw[thick,orange,rounded corners] (-5+\perturbshift,-5+\perturbshift) rectangle (5+\perturbshift,5+\perturbshift);
\node[blue,anchor=north east,inner sep=0] at (-5.2,-5.2) {$\bar\beta^0_c$};

\draw[thick,green!50!black,rounded corners] (-4,-3) rectangle (-2,3);
\node[green!50!black,anchor=west,inner sep=0] at (-4,0) {$\bar\gamma^1_c$};

\draw[thick,orange,rounded corners] (-3,2) rectangle (3,4);
\node[orange,anchor=north,inner sep=0] at (0,3.9) {$\bar\delta^1_c$};

\draw[fill=black] (-2.3,-2.3) circle (0.2);
\draw[fill=black] (-2.3,2.3) circle (0.2);
\draw[fill=black] (2.3,-2.3) circle (0.2);
\draw[fill=black] (2.3,2.3) circle (0.2);

\end{tikzpicture}
\caption{Local Heegaard diagram for a crossing and its resolutions.}
\label{fig: crossing-heegaard}
\end{figure}
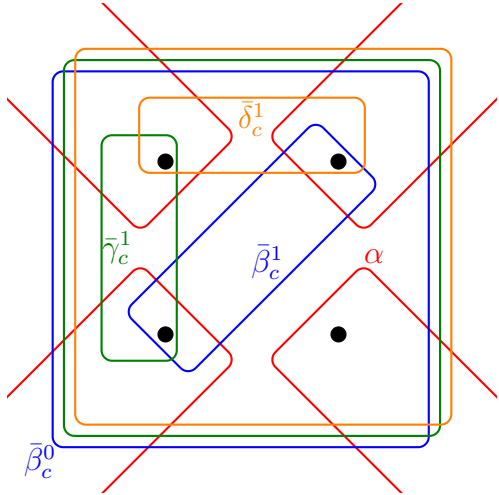 

We form a Heegaard diagram $\HH = (S^2, \bm\alpha, \{\bm\beta(v)\}_{v \in \{0,1,\infty\}^n}, P)$ that encodes $\hat\LL$ and all its resolutions as follows. Identify $S^2$ with the plane together with a point at $\infty$, using the standard orientation of the plane. For each $p \in \p$, let $\alpha_p$ be the boundary of a narrow disk that runs along the edge of $L$ on which $p$ lies, and place two basepoints inside this disk at the two ends of the edge. For each crossing $c$, choose an identification of $c$ with the crossing shown in Figure \ref{fig: resolutions}(A). Let $\bar\beta_c^0$ be the boundary of a small disk around $c$ that contains the four nearby basepoints, and let $\bar\beta_c^1$, $\bar\gamma_c^1$, and $\bar\delta_c^1$ be curves in the interior of this disk, arranged as in Figure \ref{fig: crossing-heegaard}. We also place one pair of basepoints $\{w_0,z_0\}$ in the unbounded region of the plane.

For each $v \in \{0,1,\infty\}^n$, let $\bm\beta(v)$ be the multicurve consisting of curves $\beta_{c_i}^0(v), \beta_{c_i}^1(v)$, where $\beta_{c_i}^0(v)$ is a small translate of $\bar\beta_{c_i}^0$, and $\beta_{c_i}^1(v)$ is a small Hamiltonian translate of either $\bar\beta_{c_i}^1$, $\bar\gamma_{c_i}^1$, or $\bar\delta_{c_i}^1$, according to whether $v_i=0$, $1$, or $\infty$, respectively. Unlike in \cite{BaldwinLevineSpanning}, we do not need to make any further modification to the Heegaard diagram $\HH = (S^2, \bm\alpha, \{\bm\beta(v)\}_{v \in \{0,1,\infty\}^n}, P)$ to achieve admissibility. This follows from an analysis of periodic domains as in \cite[Section 5.2]{BaldwinLevineSpanning}; details are left to the reader. Note that all of the periodic domains are linear combinations of singly periodic domains.

It is easy to see that $(S^2, \bm\alpha, \bm\beta(v), P)$ is a Heegaard diagram for $-\hat \LL_v$ (as an unoriented link). To justify the $-$ sign, consider a crossing where $v_i = \infty$. The Heegaard surface is oriented as the boundary of the $\alpha$ handlebody, which should be pictured as lying below the page. When we follow the standard procedure to produce a knot, the arc joining the upper-left and lower-right basepoints in Figure \ref{fig: crossing-heegaard} crosses an arc connecting either the lower-left or upper-right basepoint to a point outside the local diagram. The former arc is lifted above the page, and the latter is pushed below the page; thus, the crossing is the mirror of the one in Figure \ref{fig: resolutions}(A).

The first half of \cite[Lemma 5.8]{BaldwinLevineSpanning} requires finding four holomorphic triangles whose contributions cancel out; these can be seen readily in Figure \ref{fig: crossing-heegaard}. The convention of ordering the segments of the boundary of a holomorphic triangle clockwise, as in \cite[Section 8]{OSz3Manifold}, thus forces the ordering of the resolutions. Hence, we are forced to work with the mirrored links $-\LL_v$.\footnote{In \cite{BaldwinLevineSpanning}, we instead chose to reverse the convention for the $0$- and $1$-resolutions of a crossing. We also note that Manolescu \cite{ManolescuSkein} uses the standard choice of $0$- and $1$-resolutions, but implicitly uses the opposite orientation for the Heegaard surface; thus, the skein exact sequence should properly be stated for $-L$ rather than $L$ as in \eqref{eq: HFK-skein}.} The remaining results of \cite[Section 5.3]{BaldwinLevineSpanning} deriving the cube of resolutions then go through directly.

For each $v \in \{0,1\}^n$, orient $L_v$ as the boundary of the black regions. This determines a partition of $P$ into $w$ and $z$ basepoints that does not depend on $v$. As a result, the complexes $\uCFK(\bm\alpha, \bm\beta(v))$ and $\uCFK(\bm\beta(v), \bm\beta(v'))$ are equipped with absolute Alexander and Maslov gradings, as well as the algebraic structure described in Section \ref{subsec: HFK-action}.

For an immediate successor pair $v \lessdot v'$, where $v,v' \in \{0,1\}^n$, let $\Theta^0_{v,v'}, \Theta^1_{v,v'} \in \T_{\beta(v)} \cap \T_{\beta(v')}$ be the two distinguished generators: each has maximal Maslov grading among generators in the same Alexander grading, and the Maslov and Alexander gradings of $\Theta^0_{v,v'}$ are each one less than those of $\Theta^1_{v,v'}$.\footnote{In \cite{BaldwinLevineSpanning}, $\Theta^0_{v,v'}$ is called $\Theta_2^{I,I'}$.}

For a successor sequence $v^0 \lessdot \cdots \lessdot v^k$, the map
\[
f_{v^0, \cdots, v^k} \co \uCFK(\bm\alpha, \bm\beta(v^0)) \to \uCFK(\bm\alpha, \bm\beta(v^k))
\]
is defined by
\begin{equation}
f_{v^0, \cdots, v^k} (\x) = F_{\alpha, \beta(v^0), \cdots, \beta(v^k)} \left(\x \otimes \left( \Theta^0_{v^0,v^1} + \Theta^1_{v^0,v^1} \right) \otimes \cdots \otimes \left( \Theta^0_{v^{k-1},v^k} + \Theta^1_{v^{k-1},v^k} \right) \right).
\end{equation}
We may view $f_{v^0, \dots, v^k}$ as the sum of $2^k$ terms $f_{v^0, \dots, v^k}^{\bm\epsilon}$ ranging over all $\bm\epsilon = (\epsilon_1, \dots, \epsilon_k) \in \{0,1\}^k$, where
\begin{equation}
f_{v^0, \cdots, v^k}^{\bm\epsilon} (\x) = F_{\alpha, \beta(v^0), \cdots, \beta(v^k)} \left(\x \otimes \Theta^{\epsilon_1}_{v^0,v^1} \otimes \cdots \otimes \Theta^{\epsilon_k}_{v^{k-1},v^k} \right).
\end{equation}
As seen in \cite[Proposition 6.6]{BaldwinLevineSpanning}, each map $f_{v^0, \dots, v^k}^{\bm\epsilon}$ is homogeneous with respect to the Alexander and Maslov gradings; we shall compute the grading shifts of these maps below.

Let $X$ denote the direct sum
\[
X = \bigoplus_{v \in \{0,1\}^n} \uCFK(\bm\alpha, \bm\beta(v)),
\]
with differential given by
\[
D = \sum_{v^0 \lessdot \cdots \lessdot v^k} f_{v^0, \dots, v^k}.
\]
Let $\Delta$ denote the grading on $X$ obtained by shifting the delta grading on each summand $\uCFK(\bm\alpha, \bm\beta(v))$ by $(\abs{v} - n_+(L))/2$. With respect to this grading, every term in $D$ is homogeneous of degree $1$. The main theorem of \cite{BaldwinLevineSpanning}, suitably modified, states:

\begin{theorem} \label{thm: BLH}
The differential $D$ satisfies $D^2=0$, and the homology of $(X,D)$ is isomorphic to $\uHFK(-\hat\LL)$ with its delta grading.
\end{theorem}

The complex $X$ has a cubical filtration $\FF$ in the sense of Definition \ref{defn:cubical-chain-complex}, where the summand $\uCFK(\bm\alpha, \bm\beta(v))$ is defined to lie in filtration level $\abs{v}$. The differential on the associated graded complex of this filtration consists of the internal differentials on the summands $\uCFK(\bm\alpha, \bm\beta(v))$, so the $E_1$ page of the induced spectral sequence is:
\begin{equation} \label{eq:E1(X,D,F)}
E_1(X,D, \FF) \cong \bigoplus_{v \in \{0,1\}^n} \uHFK(L_v,\p).
\end{equation}
By Proposition \ref{prop: HFK-unlink}, each summand $\uHFK(L_v,\p)$ can be identified with $\robar\Lambda_{L_v,\p} \otimes \Gamma_{L_v}$. (Here $\robar\Lambda_{L_v,\p}$ and $\Gamma_{L_v}$ denote the appropriate exterior algebras over $\Z_2$.) The $d_1$ differential is the sum, taken over all $v \lessdot v' \in \{0,1\}^n$, of the maps
\[
(f_{v,v'})_* \co  \uHFK(\bm\alpha, \bm\beta(v)) \to \uHFK(\bm\alpha, \bm\beta(v')),
\]
Note that each $(f_{v,v'})_*$ splits as $(f_{v,v'}^0)_* + (f_{v,v'}^1)_*$, and these terms are each fully determined by Propositions \ref{prop: HFK-edgemaps0} and \ref{prop: HFK-edgemaps1}.
%
Thus, the homology group $E_2(X,D,\FF)$ can be computed explicitly. However, as noted in \cite[Remark 7.7]{BaldwinLevineSpanning}, $E_2(X,D,\FF)$ fails to be a knot invariant, and it is not isomorphic to a suitable multiple of Khovanov homology, as would be required to prove a rank inequality in the manner of \cite{OSzDouble}.

\subsection{The Alexander filtration}

In this section, we shall use the Alexander gradings on the complexes $\uCFK(\bm\alpha, \bm\beta(v))$ to produce a new filtration on $X$. We begin by determining the Alexander grading shifts of the maps $f_{v^0, \dots, v^k}^{\bm\epsilon}$, where $\bm\epsilon \in \{0,1\}^k$.

\begin{lemma} \label{lemma: alex-shift}
For any successor sequence $v^0 \lessdot \cdots \lessdot v^k$, and any $\bm\epsilon \in \{0,1\}^k$, the map $f_{v^0, \dots, v^k}^{\bm\epsilon}$ is homogeneous with respect to the Alexander grading of degree
\begin{equation}
A(f_{v^0, \dots, v^k}^{\bm\epsilon}) = \frac{l_{v^k} - l_{v^0} - k}{2} + \abs{\bm\epsilon}.
\end{equation}
\end{lemma}

\begin{proof}
By \cite[Proposition 6.6]{BaldwinLevineSpanning}, each map $f_{v^0, \dots, v^k}^{\bm\epsilon}$ is homogeneous, and
\[
A(f_{v^0, \dots, v^k}^{\bm\epsilon}) = A(f_{v^0, \dots, v^k}^0) + \abs{\bm\epsilon}.
\]
Thus, it suffices to compute $A(f_{v^0, \cdots, v^k}^0)$. An argument similar to that of \cite[Section 6.1]{BaldwinLevineSpanning} shows that the quantity
\[
A(f_{v^0, \dots, v^j}^0) + A(f_{v^j, \cdots v^k}^0)
\]
is independent of the choice of $j = 0, \dots, k$. Since $f_{v^0}^0 = f_{v^0}$ is simply the differential on $\uCFK(\bm\alpha, \bm\beta(v^0))$, which preserves the Alexander grading, we deduce that
\[
A(f_{v^0, \dots, v^j}^0) + A(f_{v^j, \dots, v^k}^0) = A(f_{v^0, \dots, v^k}^0).
\]
In the case $k=1$, the computation of $f_{v^0,v^1}^0$ in Proposition \ref{prop: HFK-edgemaps0} shows that $A(f_{v^0,v^1}^0)$ is $-1$ in the case of a merge and $0$ in the case of a split. Thus, in general, we have
\[
A(f_{v^0, \dots, v^k}^0) = -p,
\]
where $p$ is the number of splits in the successor sequence. Since
\[
l_{v^k} - l_{v^0} = k-2p,
\]
the result follows.
\end{proof}

Thus, if we define a new grading $\GG$ on $X$ by defining, for any homogeneous element $x \in \uCFK(\bm\alpha, \bm\beta(v))$,
\begin{equation} \label{eq: G-filt}
\GG(x) = A(x) + \frac{\abs{v} - l_v}{2},
\end{equation}
we see that $f_{v^0, \dots, v^k}$ shifts $\GG$ by
\begin{equation}
\GG(f_{v^0, \dots, v^k}^{\bm\epsilon}) = \abs{\bm\epsilon},
\end{equation}
so $\GG$ actually yields a filtration on $(X,D)$. The associated graded complex of this filtration is $X$, with differential
\[
D^0 = \sum_{v^0 \lessdot \cdots \lessdot v^k} f_{v^0, \dots, v^k}^0.
\]
The spectral sequence implies that
\begin{equation}
\rank H_\delta(X, D^0) \ge \rank \HFKtil{}^\delta(-\hat\LL).
\end{equation}

We now state a graded version of Conjecture \ref{conj: H(X,D0)-intro}:
\begin{conjecture} \label{conj: H(X,D0)}
The complex $(X,D^0)$ is quasi-isomorphic to $\KhCx(L,\p)$ via a map that identifies the $\Delta$ grading on $(X,D^0)$ with the delta grading on $\KhCx(L,\p)$ and respects the cubical filtrations on the two complexes. As a result,
\begin{equation} \label{eq: H(X,D0)}
H_\delta(X, D^0) \cong \Kh^\delta(L,\p;\Z_2),
\end{equation}
and therefore
\begin{equation} \label{eq: Kh-HFK-pointed-graded}
\rank \Kh^\delta(L,\p;\Z_2) \ge \rank \HFKtil{}^\delta(-\hat\LL;\Z_2).
\end{equation}
\end{conjecture}

The \emph{width} of any delta-graded group $H^*$ is
\[
w(H^*) = \max\{\delta \mid H^\delta \ne 0\} - \min\{\delta \mid H^\delta \ne 0\} + 1.
\]
Conjecture \ref{conj: H(X,D0)} would imply that
\[
w(\Kh(L,\p;\Z_2)) \ge w(\HFKtil(-\hat\LL;\Z_2)),
\]
and hence, using Proposition \ref{prop: Kh-reduced2} and Lemma \ref{lemma: HFK-unred}, that
\[
w(\rKh(L;\Z_2)) \ge w(\HFK(-L;\Z_2)).
\]
It is tempting to conjecture that a graded version of \eqref{eq: Kh-HFK} holds without reference to the marked points $\p$, namely that
\begin{equation} \label{eq: Kh-HFK-graded}
2^{l-1} \rank \rKh{}^\delta(L;\Z_2) \ge \rank \HFK{}^\delta(-L;\Z_2),
\end{equation}
but this does not follow directly from \eqref{eq: Kh-HFK-pointed-graded}.

The strongest evidence for Conjecture \ref{conj: H(X,D0)} comes from considering the cube filtration $\FF$ as a filtration on the complex $(X,D^0)$. The associated graded differential in the induced spectral sequence is once again given by the internal differentials on the summands $\CFK(\bm\alpha, \bm\eta(v))$, so the $E_1$ pages is again
\begin{equation} \label{eq: E1(X,D0,F)}
E_1(X,D^0, \FF) \cong \bigoplus_{v \in \{0,1\}^n} \HFKtil(-\hat\LL_v).
\end{equation}
Here, however, the $d_1$ differential is the sum of the maps
\[
(f_{v,v'}^0)_* \co  \HFKtil(-\hat\LL_v) \to \HFKtil(-\hat\LL_v').
\]
Observe that each of these maps commutes with the action of $\Lambda_\p$ on each summand of $E_1(X,D^0, \FF)$. Thus, $E_1(X,D^0,\FF)$ can be viewed as a differential $\Lambda_\p$--module. (We shall discuss this module structure in more depth below.)

With all these pieces in place, we may now state and prove a strengthened version of Theorem \ref{thm: E1iso-intro}.

\begin{theorem} \label{thm: E1iso}
The $E_1$ pages of the cubical spectral sequences associated to $(X,D^0)$ and $(\KhCx(L,\p;\Z_2), \diff)$ are isomorphic as graded chain complexes of $\Lambda_\p$--modules.
\end{theorem}

\begin{proof}
This follows directly from Propositions \ref{prop: Kh-unlink}, \ref{prop: Kh-edgemaps}, \ref{prop: HFK-unlink}, and \ref{prop: HFK-edgemaps0}. Specifically, both $E_1$ pages are isomorphic to
\[
\bigoplus_{v \in \{0,1\}^n} \robar\Lambda_{L_v,\p} \otimes \Gamma_{L_v},
\]
and the $d_1$ differentials are the maps described in Sections \ref{subsec: Kh-cobordism} and \ref{subsec: HFK-cobordism}.

To see that the gradings on the two $E_1$ pages agree, it suffices to determine the grading of
\[
1 \otimes 1 \in \robar\Lambda_{L_v,\p} \otimes \Gamma_{L_v}
\]
for each $v$. On the Khovanov side, this element is represented by $\alpha_1 \cdots \alpha_{l_v} \in \KhCx(L_v)$, for which we have:
\begin{align*}
\homgr(\alpha_1 \cdots \alpha_{l_v}) &= l_v + \abs{v}  - n_-(L)  \\
\intgr(\alpha_1 \cdots \alpha_{l_v}) &= 3l_v + \abs{v} + n_+(L) - 2n_-(L) \\
\intertext {and hence}
\deltagr(\alpha_1 \cdots \alpha_{l_v}) &= \frac{-l_v + \abs{v} - n_+(L)}{2}.
\end{align*}
On the knot Floer side, $1 \otimes 1$ is given by an element $x \in \HFKtil(-\hat\LL_v)$ whose internal Alexander and Maslov gradings are maximal, namely $A(x) = 0$ and $M(x) = \frac{l_v}{2}$, and hence
\[
\Delta(x) = A(x) - M(x) + \frac{\abs{v} - n_+(\DD)}{2} = \frac{-l_v + \abs{v} - n_+(L)}{2},
\]
as required.
\end{proof}

As noted in the Introduction, one way to prove Conjecture \ref{conj: H(X,D0)} would be to construct a filtered chain map
\[
\Phi\co (X,D^0) \to \KhCx(\LL)
\]
that induces the isomorphism of Theorem \ref{thm: E1iso} on the $E_1$ page of the spectral sequences. To write down such a map explicitly, one would presumably need a complete description of $(X,D^0)$, which would require determining the counts of holomorphic polygons corresponding to every face in the cube of resolutions.

\subsection{Additional algebraic structure}

As further evidence for Conjecture \ref{conj: H(X,D0)}, we now show that the homology $H_*(X,D^0)$ has commuting actions of the exterior algebras $\Lambda_\p$ and $\Gamma_\p$, akin to the structures on $\Kh(L,\p;\Z_2)$ described in Section \ref{subsec: additional}.

First, we consider the action of $\Lambda_\p$. Just as in Remark \ref{rmk: A-infty}, this arises from an $A_\infty$--module structure on the chain level. For any linear combination $\a \in (\Z_{\ge 0})^\p$, and any distinct sets of attaching curves $\bm\eta_0, \dots, \bm\eta_k$ (either $\bm\alpha$ or $\bm\beta(v)$ for $v \in \{0,1\}^n$), define a map
\[
\Psi^\a_{\eta_0, \dots, \eta_k} \co \uCFK(\bm\eta_0, \bm\eta_1) \otimes \cdots \otimes \uCFK(\bm\eta_{k-1}, \bm\eta_k) \to \uCFK(\bm\eta_0, \bm\eta_k)
\]
that counts holomorphic polygons whose multiplicities at the $z$ basepoints are specified by $\a$, and whose multiplicities at the $w$ basepoints are all zero. That is, for $\x_i \in \T_{\eta_{i-1}} \cap \T_{\eta_i}$, define
\begin{equation} \label{eq: Psi-polygons}
\Psi^\a_{\eta_0, \dots, \eta_k} (\x_1 \otimes \cdots \otimes \x_k) = \\
\sum_{\y \in \T_{\eta_0} \cap \T_{\eta_k}} \sum_{\substack{ \psi \in \pi_2(\x_1, \cdots, \x_k, \y) \\
\mu(\psi) = 2-k \\ \mults_\z(\psi) = \a  \\ \mults_\w(\psi) = 0 }} \#(\MM(\psi)) \y.
\end{equation}
The case $k=1$ corresponds to the maps from Remark \ref{rmk: A-infty}. A standard degeneration argument shows that these maps satisfy an $A_\infty$ relation: for any $\x_i \in \T_{\eta_{i-1}} \cap \T_{\eta_i}$, we have:
\begin{multline} \label{eq: Psi-polygons-relation}
\sum_{1 \le i \le j \le k} \sum_{\a = \b + \c} \Psi^\b_{\eta_0, \dots, \eta_i, \eta_j, \dots, \eta_k} \big( \x_1 \otimes \dots \otimes \x_{i-1} \otimes 
\Psi^\c_{\eta_{i-1}, \dots, \eta_j} \left( \x_i \otimes \cdots \otimes \x_j \right) \\ \otimes \x_{j+1} \otimes \cdots \otimes \x_k \big) = 0.
\end{multline}

\begin{lemma} \label{lemma: Psi(allTheta0)}
For any $k\ge 1$, any successor sequence $v^0 \lessdot \cdots \lessdot v^k$ in $\{0,1\}^n$, and any nonzero vector $\a \in \Z_{\ge 0}^m$, we have
\begin{equation} \label{eq: Psi(allTheta0)}
\Psi^\a_{\beta(v^0), \dots, \beta(v^k)}(\Theta^0_{v^0,v^1} \otimes \dots \otimes \Theta^0_{v^{k-1},v^k}) = 0.
\end{equation}
\end{lemma}

\begin{proof}
In the case where $k=1$, this follows directly from the fact that $\Theta^0_{v^0,v^1}$ is in the minimal Alexander grading and that $\Psi^\a_{\beta(v^0),\beta(v^1)}$ decreases Alexander grading by $\abs{\a}$. Hence, we may assume that $k>1$.

Just as in \cite[Section 5.1]{BaldwinLevineSpanning}, the summand of $\CFKtil(\bm\beta(v^0), \bm\beta(v^k))$ in minimal $\delta$ grading is generated by $2^k$ preferred generators. We may index these by $\Theta^{\bm\epsilon}_{v^0,v^k}$ for $\bm\epsilon = (\epsilon_1, \dots, \epsilon_k) \in \{0,1\}^k$, such that for each $i=1, \dots, k$, $\Theta^{\bm\epsilon}_{v^0,v^k}$ is ``close'' to $\Theta^{\epsilon_i}_{v^{i-1},v^i}$. By considering the bigons containing the $z$ basepoints, we see that $\Theta^{\bm0}_{v^0,v^k}$ has minimal Alexander grading among these generators.

Following the proofs of Lemma 5.8 and Proposition 5.9 from \cite{BaldwinLevineSpanning}, we may find a homotopy class
\[
\psi \in \pi_2(\Theta^{0}_{v^0,v^1}, \dots, \Theta^{0}_{v^{k-1},v^k}, \Theta^{\bm0}_{v^0,v^k})
\]
that avoids all the basepoints. If $\psi' \in \pi_2(\Theta^{0}_{v^0,v^1}, \dots, \Theta^{0}_{v^{k-1},v^k}, \Theta')$ is some other class that contributes to
\[
\Psi^\a_{\beta(v^0), \dots, \beta(v^k)}(\Theta^0_{v^0,v^1} \otimes \dots \otimes \Theta^0_{v^{k-1},v^k}),
\]
then the difference of the domains of $\psi'$ and $\psi$, plus some periodic domain $P$, gives a class $\phi \in \pi_2(\Theta^{\bm0}_{v^0,v^k}, \Theta')$. Then
\[
A(\Theta^{\bm0}_{v^0,v^k}) - A(\Theta') = n_\z(\phi) - n_\w(\phi) = n_\z(\psi') - n_\w(\psi') = \abs{\bm\a},
\]
which would violate the minimality of $A(\Theta^{\bm0}_{v^0,v^k})$ unless $\bm\a=0$.
\end{proof}

For each successor sequence $v^0 \lessdot \dots \lessdot v^k$, we then define a map
\[
\Psi^\a_{v^0, \dots, v^k} \co \uCFK(\bm\alpha, \bm\eta(v^0)) \to \uCFK(\bm\alpha, \bm\eta(v^k))
\]
by
\[
\Psi^\a_{v^0, \dots, v^k}(\x) = \Psi^\a_{\alpha, \beta(v^0), \dots, \beta(v^k)} \left( \x \otimes \Theta^0_{v^0,v^1} \otimes \cdots \otimes \Theta^0_{v^{k-1},v^k} \right).
\]
We then let $\Psi^\a \co X \to X$ be the sum of all the component maps $\psi^\a_{v^0, \dots, v^k}$, taken over all successor sequences. In particular, $\Psi^0$ is simply the differential $D^0$. The relation \eqref{eq: Psi-polygons-relation} together with Lemma \ref{lemma: Psi(allTheta0)} then implies that
\[
\sum_{\a = \b + \c} \Psi^\b \circ \Psi^\c = 0.
\]
Just as in Remark \ref{rmk: A-infty}, these maps give $(X,D^0)$ the structure of an $A_\infty$ module over $\Lambda_\p$ via homological perturbation theory. More concretely, the maps $\Psi^{\e_i}$ are filtered chain maps which commute up to chain homotopy and whose squares are null-homotopic. Both the homology $H_*(X,D^0)$ and all of the pages of the spectral sequence coming from the cubical filtration on $(X,D^0)$ acquire the structure of (honest) $\Lambda_\p$--modules, where $y_p \in \Lambda_\p$ acts by $\Psi_p$.

Next, we consider the $H_1$ action. For any immersed $1$-chain $\zeta$ with endpoints in $\w \cup \z$ and in general position with respect to the Heegaard diagram, and any Whitney polygon $\psi$, there is a well-defined intersection number $\zeta(\psi) = \#(\zeta_p \cdot \partial_\alpha(\psi))$. Hence, for any successor sequence $v_0 \lessdot \dots \lessdot v_k$, we may define a map
\[
A^{\zeta}_{v^0, \dots, v^k} \co \CFKtil(\bm\alpha, \bm\beta(v^0)) \to \CFKtil(\bm\alpha, \bm\beta(v^k))
\]
by
\begin{equation} \label{eq: Azeta-polygon}
A^\zeta_{v^1, \dots, v^k} (\x) = \sum_{\y \in \T_\alpha \cap \T_{\beta(v^k)}} \sum_{ \substack{\psi \in \pi_2(\x, \Theta^0_{v^1,v^2}, \dots, \Theta^0_{v^{k-1},v^k}, \y) \\ \mu(\psi)=2-k \\ n_\w(\psi) = n_\z(\psi) = 0 }} \#\widehat{\MM}(\psi) \, \zeta(\psi) \, \x.
\end{equation}
(In particular, when $k=1$, the map $A^\zeta_{v^1}$ induces the action of $[\zeta] \in H_1(S^3, \hat L_{v^1})$ on $\HFKtil(\LL_{v^1})$.) We then define a filtered map $A^\zeta \co X \to X$ as the sum of the $A^\zeta_{v^0,\dots,v^k}$ ranging over all successor sequences. It is straightforward to show that $A^\zeta$ commutes with $D^0$ and that its square is null-homologous. (Compare Hedden--Ni \cite[Lemma 4.4]{HN-kh-unlink-detection}.) Moreover, if $\zeta_1$ and $\zeta_2$ are homologous \emph{rel endpoints}, then $A^{\zeta_1}$ and $A^{\zeta_2}$ are filtered chain homotopic, via an adaptation of the argument in \cite[Lemma 2.4]{NiActions}.

In particular, for each $p \in \p$, let $\zeta_p$ be an arc from $w_0$ to $z_p$. The maps $A^{\zeta_p}$ thus make $H_*(X,D^0)$ and all of the pages of the spectral sequence into $\Gamma_\p$--modules. By construction, the module structure on the $E_1$ page respects the direct sum decomposition \eqref{eq: E1(X,D0,F)} and agrees with the $\Lambda^* H_1(S^3,\hat L_v)$--module structure of each summand $\HFKtil(\bm\LL_v)$.

Finally, generalizing \eqref{eq: Psi-Azeta}, note that $A^{\zeta_p}$ commutes up to filtered chain homotopy with $\Psi^{p'}$ (for any $p, p' \in \p$); the chain homotopy counts Whitney polygons $\psi$ with $n_z(\psi_{p})=1$, weighted by $\zeta_{p'}(\psi)$. Thus, the actions of $\Lambda_\p$ and $\Gamma_\p$ on $H_*(X,D^0)$ commute, as do the actions on each of page of the spectral sequence.

\bibliographystyle{amsalpha}
\bibliography{bibliography,SSbibfile}

\end{document}